\documentclass[11pt]{article}
\usepackage{graphicx} 

\title{On the Kobayashi-Hitchin correspondence\\ for K\"{a}hler currents}
\author{Satoshi Jinnouchi \thanks{Department of Mathematics, Graduate School of Science, Osaka University,
1-1, Machikaneyama-cho, Toyonaka, Osaka 560-0043, Japan.
email:{{\tt u122988d@ecs.osaka-u.ac.jp}},
email:{{\tt 20160312sti@gmail.com}}}}
\date{August 2025}

\usepackage{amssymb,amsmath,amsthm}    
\usepackage[abbrev]{amsrefs} 
\usepackage{cite}    
\usepackage{mathrsfs}
\usepackage[bbgreekl]{mathbbol}
\usepackage{comment}
\usepackage{color}
\usepackage{tikz}
\usepackage{mathtools}
\usepackage{appendix}

\newtheorem{theo}{Theorem}[section]
\newtheorem{lemm}[theo]{Lemma}
\newtheorem{corr}[theo]{Corollary}
\newtheorem{prop}[theo]{Proposition}

\numberwithin{equation}{section}

\theoremstyle{definition}
\newtheorem{defi}[theo]{Definition}

\newtheorem{rema}[theo]{Remark}
\newtheorem{assu}[theo]{Assumption}
\newtheorem{claim}[theo]{Claim}

\allowdisplaybreaks

\newcommand{\rk}{{\rm{rk}}}
\newcommand{\Supp}{{\rm{Supp}}}
\newcommand{\Amp}{{\rm{Amp}}}
\newcommand{\sing}{{\rm{sing}}}
\newcommand{\Sing}{{\rm{Sing}}}
\newcommand{\Exc}{{\rm{Exc}}}

\newcommand{\codim}{{\rm{codim}}}
\newcommand{\Id}{{\rm{Id}}}

\newcommand{\loc}{{\rm{loc}}}

\newcommand{\Hom}{{\rm{Hom}}}

\newcommand{\Tor}{{\rm{Tor}}}
\newcommand{\Herm}{{\rm{Herm}}}
\newcommand{\End}{{\rm{End}}}

\newcommand{\Ric}{\rm{Ric}}
\newcommand{\E}{\mathcal{E}}
\newcommand{\Tr}{{\rm{Tr}}}

\newcommand{\delbar}{\bar{\partial}}

\usepackage[colorlinks=true, linkcolor=blue, citecolor=blue, urlcolor=cyan]{hyperref}

\begin{document}
\date{\empty}
\maketitle
\begin{abstract}
In this paper, we show that if a holomorphic vector bundle is slope polystable with respect to a K\"{a}hler class, then it admits a Hermitian-Yang-Mills metric with respect to a suitable K\"{a}hler current with singularities in higher codimension which represents the K\"{a}hler class. Most parts of the proof remains valid for closed positive $(1,1)$-currents representing a nef and big class.
\end{abstract}
\tableofcontents

\section{Introduction}
The classical version of {\it{the Kobayashi-Hitchin correspondence}} states that a holomorphic vector bundle $E$ over a compact K\"{a}hler manifold $(X,\omega)$ is $\{\omega\}^{n-1}$-slope polystable if and only if $E$ admits an $\omega$-Hermitian-Yang-Mills metric. The correspondence was proved by Donaldson \cite{Don87} and Uhlenbeck and Yau \cite{UY} on compact K\"{a}hler manifolds. Bando and Siu \cite{BS} extended the correspondence for reflexive sheaves on compact K\"{a}hler manifolds by introducing the notion of admissible HYM metrics.
Recently the correspondence has been established for singular K\"{a}hler varieties endowed with smooth K\"{a}hler metrics \cite{Chen}. If K\"{a}hler metrics have at most orbifold singularities, related results were obtained in \cite{Faulk22} and \cite{CGNPPW23}.

The main result of our paper is a generalization of the Kobayashi-Hitchin correspondence for K\"{a}hler currents representing  K\"{a}hler classes.
The notion of a Hermitian-Yang-Mills metric (connection) for a current is defined as follows:
\begin{defi}
Let $X$ be a compact K\"{a}hler manifold and $T$ be a closed positive $(1,1)$-current which is smooth K\"{a}hler on a Zariski open set $\Omega$. Let $({E},h_0)$ be a complex hermitian vector bundle on $X$.\\
$(1)$ A {\it{$T$-admissible Hermitian-Yang-Mills (HYM, in short) connection}} on $(E, h_0)$ is a smooth $h_0$-connection $\nabla$ on $E|_{\Omega}$ which satisfies the followings:
\begin{itemize}
\item (integrability) $\delbar^{\nabla}\circ\delbar^{\nabla}=0$.
\item ($T$-HYM equation) $\sqrt{-1}\Lambda_TF_{\nabla}=\lambda\Id$ on $\Omega$ where $\lambda$ is a real number and
\item ($T$-admissible condition) $\int_{\Omega}|F_{\nabla}|_T^2T^n<\infty$.
\end{itemize}
\noindent We call $\lambda$ as the $T$-HYM constant of $\nabla$.\\
\noindent $(2)$ A $T$-admissible HYM metric on a holomorphic vector bundle $E$ is a smooth hermitian metric $h$ on $E|_{\Omega}$ such that the Chern connection $\nabla_h$ of $h$ is  a $T$-admissible HYM connection of $(E, h)$.
\end{defi}
The main result of the present paper is the following.
\begin{theo}[= Theorem \ref{main thm thm}]\label{main thm}
Let $(Y,\omega_Y)$ be a compact K\"{a}hler manifold.
Let $\omega$ be the solution to the complex Monge-Amp\`{e}re equation $\omega^n=e^F\omega_Y^n$, where $F$ satisfies the following conditions:
\begin{itemize}
\item there is an analytic subset $Z\subset Y$ with $\codim Z\ge 2$ such that $F\in C^{\infty}(Y\setminus Z)$ and
\item there is a constant $0\le a \le 1/2$ such that $F=-a\log(|f_1|^2+\cdots+|f_l|^2)+O(1)$ locally where $(f_1,\ldots,f_l)$ is a local defining functions of $Z$. 
\end{itemize}
Then, if a holomorphic vector bundle $F$ on $Y$ is $\{\omega\}^{n-1}$-slope polystable, then $F$ admits an $\omega$-admissible HYM metric.
\end{theo}
\begin{rema}
As explained below, the present paper basically works on $X$ with $\pi^*\omega$ where $\pi:X\to Y$ is a resolution along $Z$ for the proof of Theorem \ref{main thm}.
Although $Y$ in Theorem \ref{main thm} is smooth, the above idea is essentially the same as working on a log smooth pair $(X,D)$ in the context of the studying singular K\"{a}hler-Einstein metrics on a normal variety $Y$. If $K_X+D$ is ample and the solution $\omega$ to the complex Monge-Amp\`{e}re equation defines a singular K\"{a}hler-Einstein metric on $(X,D)$, then it has been proved that $\omega$ has conical singularities (e.g. \cite{Gue16}, \cite{GW16}). And the Kobayashi-Hitchin correspondence on a compact K\"{a}hler manifold with conical K\"{a}hler metrics has been studied (e.g. \cite{Li00}, see also \cite{JL25}). However, even if $\{\omega\}=c_1(K_Y)$, the pull-back class $\{\pi^*\omega\}=c_1(K_X+D)$ is not ample where $D=-a\Exc(\pi)$ with $a>0$. Thus the existence of a model metric in this setting remains unknown. Moreover, the approach of this paper does not rely on the existence of model metrics. In fact, most parts of the proof in this paper continue to hold for closed positive $(1,1)$-currents representing nef and big classes (Assumption \ref{assumption intro}).
\end{rema}

By \cite{CCHTST25}, we can see that $\omega$ is a K\"{a}hler current on $Y$ and there exists a modification $\pi:X\to Y$ along $Z$ such that $\pi^*\omega$ is approximated by a sequence of smooth K\"{a}hler metrics $\omega_i$ on $X$.
Since the slope stability is invariant under modifications (Lemma \ref{invariance}), it suffices to find a $\pi^*\omega$-admissible HYM metric on $E:=\pi^*F$. 
We focus on the fact that $\pi^*\{\omega\}$ is nef and big and $\pi^*\omega$ is a closed positive $(1,1)$-current which satisfies the complex Monge-Amp\`{e}re equation $(\pi^*\omega)^n=e^f|s_D|^q\omega_X^n$ with $q\ge 1$ (Lemma \ref{regularity of omega}). Then we also consider the following more general setting. 
If we assume that Lemma \ref{uniform estimate intro} holds for positive currents in Assumption \ref{assumption intro}, then the rest of the proof remains valid for such a current.
\begin{assu}\label{assumption intro}
We will consider the following objects:
\begin{enumerate}
\item a compact K\"{a}hler manifold $(X,\omega_0)$,
\item a nef and big class $\alpha$ on $X$ such that $D=E_{nK}(\alpha)$ is a snc divisor,
\item an $\alpha^{n-1}$-slope stable holomorphic hermitian vector bundle $(E,h_0)$ on $X$,
\item a closed positive $(1,1)$-current $T$ in $\alpha$ which satisfies $\langle T^n\rangle = e^f|s_D|^{2q}\omega_0^n$ with $q\ge 0$.
\end{enumerate}
\end{assu}
\noindent We can always find $T\in \alpha$ which satisfies the above condition $(4)$ (Lemma \ref{positive current nef big}).

We now describe the strategy of the proof.
Our approach is inspired by the recent work \cite{Chen} who proved the correspondence on compact normal varieties with smooth K\"{a}hler metrics. 
The major difference from \cite{Chen} is that the analysis is carried out on $X$, rather than on $Y$, by using a positive current $\pi^*\omega$ which is singular and degenerates along the exceptional divisor $D=E_{nK}(\pi^*\omega)$. The codimension 1 singularities of $\pi^*\omega$ and the fact that the class $\{\pi^*\omega\}$ is not K\"{a}hler give new difficulaties, explained after Theorem \ref{toward nef and big class}, which are not present in \cite{Chen}.

At first, we establish the following general result after \cite{Chen}.
\begin{theo}[=Theorem \ref{limiting objects}]\label{toward nef and big class}
Let $X$ be a compact K\"{a}hler manifold and $\alpha$ be a nef and big class on $X$ such that $D=E_{nK}(\alpha)$ is a snc divisor. Let $T$ be a closed positive $(1,1)$-current in $\alpha$ which satisfies Assumption \ref{assumption intro}.
Let $E$ be a holomorphic vector bundle on $X$. If $E$ is $\alpha^{n-1}$-slope stable, then  there exists the following data:
\begin{itemize}
\item  An analytic subset $\Sigma$ in $X\setminus D$ of $\codim \Sigma \ge 2$. We denote by $Z=D\cup\Sigma$.
\item  A  holomorphic vector bundle $E_{\infty}$ on $X\setminus Z$ such that the underlying complex vector bundle of $E_{\infty}$ coincides with $E|_{X\setminus Z}$ {\rm{(\cite{CW22} Corollary 48)}},
\item  A $T$-admissible HYM connection $\nabla_{\infty}$ over $E_{\infty}$ and
\item nontrivial sheaf morphisms  $\Phi_{\infty}:E_{\infty}\to E|_{X\setminus Z}$ and $\Psi_{\infty}: E|_{X\setminus Z}\to E_{\infty}$.
\end{itemize}
\end{theo}
\noindent 
For the proof of Theorem \ref{main thm}, we need to prove that $\Phi_{\infty}:E_{\infty}\to E|_{X\setminus Z}$ and $\Psi_{\infty}: E|_{X\setminus Z}\to E_{\infty}$ are isomorphic. On compact normal varieties with smooth K\"{a}hler metrics, Xuemiao Chen \cite{Chen} resolved the problem by proving that $E_{\infty}$ extends to the whole of $X$ as an application of Siu's Hartogs type extension \cite{Siu75}, in that cases $Z=X_{\sing}\cup \Sigma$ whose hausdorff codimension is at least 4. Main technical difficulties of the present paper arise from this point. Indeed, we cannot apply Siu's extension theorem since the subvariety $D=E_{nK}(\alpha)$, on which $E_{\infty}$ is not defined, has complex codimension 1. 
To resolve this issue, we establish the following global estimates, which valid for $T=\pi^*\omega$ in Theorem \ref{main thm}.
The Sobolev inequality in \cite{GPSS23} (refer to Lemma \ref{sob ineq}) plays an important role.
\begin{lemm}[=Corollary \ref{L-infty estimate of limit object}, Proposition \ref{uniform esti of diff}]\label{uniform estimate intro}
Let $T=\pi^*\omega$ in Theorem \ref{main thm}.
Let $\Phi_{\infty}:E_{\infty}\to E|_{X\setminus Z}$ and $\Psi_{\infty}:E|_{X\setminus Z}\to E_{\infty}$ be nonzero morphisms in Theorem \ref{toward nef and big class}. Let $s_D$ be a defining section of the $\pi$-exceptional divisor $D$. Then there exists a constant $m>0$ such that the followings hold:
\begin{itemize}
\item (Corollary \ref{L-infty estimate of limit object}) $\sup_{X\setminus Z}|\Psi_{\infty}|<\infty$ and $\sup_{X\setminus Z}|\Phi_{\infty}|<\infty$.
\item (Proposition \ref{uniform esti of diff}) $\int_{X\setminus Z}|\nabla_{\infty}(s_D^{2m}\Psi_{\infty})|_T^2T^n<\infty$ and $\int_{X\setminus Z}|\nabla_{\infty}(s_D^{2m}\Phi_{\infty})|_T^2T^n<\infty$.
\end{itemize}
Here the norms of endomorphisms are defined by smooth hermitian metrics on each vector bundles (refer to Proposition \ref{uniform esti of diff} and Corollary \ref{L-infty estimate of limit object} for more precise statements).
\end{lemm}
Assume that Lemma \ref{uniform estimate intro} holds for a closed positive $(1,1)$-current $T$ in a nef and big class $\alpha$ in Assumption \ref{assumption intro}, and prove the existence of a $T$-admissible HYM metric on an $\alpha^{n-1}$-slope stable vector bundle $E$, which implies Theorem \ref{main thm}.
The proof of Theorem \ref{main thm} is divided in 3-steps: [1] (Proposition \ref{KH corr on nK locus}) We prove the existence of a $T$-admissible HYM metric on $\mathcal{O}(ND)$ where $N>0$ is a large constant and $D=E_{nK}(\alpha)$. Next we show [2] (Theorem \ref{KH corr for rank1}) the existence of a $T$-admissible HYM metric on any holomorphic line bundle. Finally we show [3] (Theorem \ref{KH corr nef and big}) any $\alpha^{n-1}$-slope stable holomorphic vector bundle admits a $T$-admissible HYM metric. All these steps are carried out by proving that $\Psi_{\infty}:E|_{X\setminus Z}\to E_{\infty}$ is isomorphic. In the rank 1 cases [1] and [2], the proofs rely on the maximum principle together with Lemma \ref{uniform estimate intro}. The general case [3] is proved by an argument of contradiction: assuming that $\Psi_{\infty}$ is not of full rank, its kernel produces a destabilizing subsheaf of $E$, contradicting the stability of $E$. The main difficulty here is that the kernel of $\Psi_{\infty}$ is a priori defined only on $X\setminus Z$. Hence we first construct a coherent extension of the kernel over $X$ and then apply the maximum principle to conclude that $\Psi_{\infty}$ is indeed isomorphic.

\section*{Acknowledgment}
The author would like to thank his supervisor Ryushi Goto for his advice and critical comments. He also thanks the organizers of the workshop ``Canonical Metrics on K\"{a}hler Manifolds and Related Topics 6'' for giving me the opportunity to study and discuss the related topics. The author would like to thank Prof. Yuji Odaka for helpful discussions, and to Prof. Yoshinori Hashimoto for suggesting the Uhlenbeck compactness to solve the problem and discussions. The author would also thank Prof. Masataka Iwai for the advice and comments.
The author is grateful to Prof. Yuta Watanebe and Rei Murakami for fruitful discussions and important comments.

\section{Preliminaries}
\subsection{closed positive $(1,1)$-currents and big cohomology classes}
The main references of this subsection are \cite{Dem1} and \cite{Dem2}.
Let $X$ be a compact K\"{a}hler manifold of complex dimension $n$. 
The vector space of smooth $(p,q)$-forms with local $C^{\infty}$ topology is denoted by $A^{p,q}(X)$. Then a bidegree $(p,q)$-current is a continuous linear functional $T$ on $A^{n-p,n-q}(X)$. An $(n,n)$-current $T$ on $X$ is positive if, for any nonnegative function $f$ on $X$, the pairing $T(f)$ is nonnegative. And a $(p,p)$-current $T$ is positive if $T$ is real and, for any $(1,0)$-forms $\eta_1,\ldots,\eta_{n-p}$, the $(n,n)$-current $T\wedge(\sqrt{-1}\eta_1\wedge\overline{\eta_1})\wedge\cdots\wedge(\sqrt{-1}\eta_{n-p}\wedge\overline{\eta_{n-p}})$ is positive.
Any positive $(p,p)$-current is locally expressed by a differential form with coefficients of complex Radon measures. Hence the pairing $T(\eta)$ is denoted by the integral $\int_XT\wedge\eta$. 
The vector space of $(p,q)$-currents on $X$ with $\delbar$ defines a complex. The cohomology group of the complex of currents is isomorphic to the Dolbeault cohomology group. Hence, the following definition make sense. We only note the definition for compact K\"{a}hler manifolds. The reader can refer to \cite{IJZ25} for the definition for singular and non K\"{a}hler spaces.
\begin{defi}\label{big defi}
Let $(X,\omega)$ be a compact K\"{a}hler manifold and $\alpha\in H^{1,1}(X,\mathbb{R})$ be a cohomology class. Then we define as follows.
\begin{enumerate}
\item The class $\alpha$ is pseudo-effective if $\alpha$ is represented by a closed positive $(1,1)$-current.
\item The class $\alpha$ is big if $\alpha$ is represented by a K\"{a}hler current. Here a K\"{a}hler current is a closed positive $(1,1)$-current $T$ such that $T\ge \varepsilon\omega$.
\item The class $\alpha$ is nef if, for any $\varepsilon>0$, there is a smooth $(1,1)$-form $\alpha_{\varepsilon}$ in $\alpha$ such that $\alpha_{\varepsilon}\ge -\varepsilon\omega$.
\end{enumerate}
\end{defi}
Let $\theta$ be a smooth closed $(1,1)$-form on $X$. Then a $\theta$-(resp.strongly) plurisubharmonic function is a locally integrable and upper semicontinuous function $\varphi$ on $X$ such that $\theta+dd^c\varphi$ is a positive $(1,1)$-current (resp. K\"{a}hler current). On a compact K\"{a}hler manifold $X$, any closed positive $(1,1)$-current (resp.K\"{a}hler current) in a pseudo-effective class (resp.big class) $\alpha$ is always expressed by $T=\theta+dd^c\varphi$ where $\theta\in \alpha$ is smooth and $\varphi$ is a $\theta$-(resp. strongly) psh. We call the $\varphi$ by a potential of $T$.

\subsection{nonpluripolar products and positive products}\label{nonpluripolar product section}
The results of this section are of \cite{Bou04} and \cite{BEGZ}. For singular settings, we can refer to \cite{IJZ25}. Let $X$ be a compact K\"{a}hler manifold.
If potentials of closed positive $(1,1)$-currents $T_1,\ldots,T_k$ are locally bounded, then Bedford-Taylor \cite{BT} proved that the wedge product $T_1\wedge\ldots\wedge T_k$ is a well-defined closed positive $(k,k)$-current. The notion of nonpluripolar product is a generalization of Bedford-Taylor's products for general closed positive $(1,1)$-currents.\\
Let $\theta$ be a smooth closed $(1,1)$-form, $\varphi_1$ and $\varphi_2$ be $\theta$-plurisubharmonic functions on $X$. Then we define that $\varphi_1$ is {\it{less singular}} than $\varphi_2$ if $\varphi_2\le \varphi_1 + O(1)$. 
Let $\alpha$ be a pseudo-effective class on $X$. For closed positive $(1,1)$-currents $T_1=\theta+dd^c\varphi_1$ and $T_2=\theta+dd^c\varphi_2$ in $\alpha$, we define that $T_1$ is {\it{less singular}} than $T_2$ if $\varphi_1$ is less singular than $\varphi_2$.
A closed positive $(1,1)$-current $T$ in $\alpha$ is said to has {\it{minimal singularities}} if $T$ is less singular than any other closed positive $(1,1)$-currents in $\alpha$.
 We often denote by $T_{\min}$ a closed positive $(1,1)$ current in $\alpha$ with minimal singularities. 
\begin{defi}[\cite{Bou04}]
Let $X$ be a compact K\"{a}hler manifold and $\alpha$ be a pseudo-effective class. The {\it{ample locus}} of $\alpha$, denoted by $\Amp(\alpha)$, is defined as follows:
$$
\Amp(\alpha)\\
:=\{x\in X \mid \hbox{$\alpha$ admits a K\"{a}hler current which is smooth around $x$.} \}.
$$
The {\it{nonK\"{a}hler locus}} is defined as $E_{nK}(\alpha):=X\setminus \Amp(\alpha)$.
\end{defi}
\noindent The ample locus of $\alpha$ is non empty if and only if $\alpha$ is big. \\
The {\it{Lelong number}} of a plurisubharmonic function $\varphi$ on $\Omega\subset \mathbb{C}^N$ at $x\in \Omega$ is defined as
$$
\nu(\varphi, x):=\liminf_{z\to x}\frac{\varphi(z)}{\log|z-x|}.
$$
Then the {\it{Lelong number of a closed positive $(1,1)$-current $T$ at $x$}} is defined as $\nu(T, x):=\nu(\varphi,x)$ where $\varphi$ is a potential of $T$.
The {\it{minimal multiplicity}} of a big class $\alpha$ is defined \cite{Bou04} as $\nu(\alpha, x):=\nu(T_{\min}, x)$. In \cite{Bou04}, the following characterization of the nefness of a big class was proved:
\begin{prop}[\cite{Bou04}]
Let $X$ be a compact K\"{a}hler manifold and $\alpha$ be a big class on $X$. Then $\alpha$ is nef if and only if $\nu(\alpha, x)=0$ for any $x\in X$. 
\end{prop}
\noindent We remark that above proposition does {\it{not}} imply that a closed positive $(1,1)$-current with minimal singularities in a nef and big class has locally bounded potentials. We refer to \cite{BEGZ} Example 5.4 for an example of a nef and big class whose closed positive $(1,1)$-currents with minimal singularities admit poles.

Boucksom \cite{Bou04} proved that the following property:
\begin{prop}[\cite{Bou04}]\label{kahler current nonkahler locus}
Let $X$ be a compact K\"{a}hler manifold and $\alpha$ be a big class on $X$. Then there exists a K\"{a}hler current $T$ with analytic singularities in $\alpha$ such that $E_+(T):=\{x\in X\mid \nu(T, x)>0\}=E_{nK}(\alpha)$.
In particular, the nonK\"{a}hler locus $E_{nK}(\alpha)$ is a proper analytic subset of $X$.
\end{prop}
By applying the Siu decomposition to $T$ in Proposition \ref{kahler current nonkahler locus}, we obtain the following lemma which is a special type of approximate Zariski decompositions.
\begin{lemm}
Let $X$ be a compact K\"{a}hler manifold and $\alpha$ be a nef and big class such that $D=E_{nK}(\alpha)$ is a snc divisor. Then there exists a K\"{a}hler metric $\omega_0$ and an effective $\mathbb{R}$-divisor $D'$ such that $\alpha=\{\omega_0\}+[D']$ holds and $\Supp(D')=E_{nK}(\alpha)$.
\end{lemm}
Since a big class $\alpha$ contains a K\"{a}hler current $T$ which is smooth K\"{a}hler on $\Amp(\alpha)$, we obtain the important observation: {\it{any local potentials of a closed positive $(1,1)$-current $T_{\min}\in\alpha$ with minimal singularities are locally bounded on $\Amp(\alpha)$}}.
We say that a closed positive $(1,1)$-current $T$ on $X$ has a {\it{small unbounded locus}} if the unbounded locus of the local potentials of $T$ is contained in a proper analytic subset in $X$. In particular, any closed positive $(1,1)$-current with minimal singularities in a big class has small unbounded locus. Then the nonpluripolar product of a closed positive $(1,1)$-current with small unbounded locus is defined as follows. The readers can refer to \cite{BEGZ} for more general definition:
\begin{defi}[\cite{BEGZ}]\label{nonpluripolar product}
Let $X$ be a compact K\"{a}hler manifold. Let $T$ be a closed positive $(1,1)$-current whose unbounded locus is contained in a proper analytic subset $V$ in $X$. Then, the $p$-th {\it{nonpluripolar product}} of $T$ is a closed positive $(p,p)$-current on $X$ defined as follows:
$$
\langle T^p\rangle:=\mathbb{1}_{X\setminus V}T^p,
$$
here $\mathbb{1}_{X\setminus V}$ is the defining function of $X\setminus V$. 
\end{defi}
\begin{defi}[\cite{BEGZ}]
Let $X$ be a compact K\"{a}hler manifold and $\alpha$ be a big class on $X$. The {\it{positive product}} of $\alpha$ is a cohomology class defined as follows:
$$
\langle\alpha^p\rangle=\{\langle T_{\min}^p\rangle\}
$$ 
where $T_{\min}$ is a closed positive $(1,1)$-current with minimal singularities in $\alpha$.
\end{defi}
\noindent If $\alpha$ is nef and big, then $\langle\alpha^p\rangle=\alpha^p$ \cite{BEGZ}.
The following result of Collins-Tosatti \cite{CT15} is repeatedly used in the present paper:
\begin{prop}[\cite{CT15}]\label{nonkahler null}
Let $X$ be a compact K\"{a}hler manifold and $\alpha$ be a nef and big class. Let $\pi:\widehat{X}\to X$ be a modification of $X$ so that $E_{nK}(\pi^*\alpha)=\pi^{-1}(E_{nK}(\alpha))\cup \Exc(\pi)$ is a snc divisor. Then, for any divisor $D$ whose support set is contained in $E_{nK}(\pi^*\alpha)$, the following holds:
$$
(\pi^*\alpha)^{n-1}\cdot [D]=0.
$$
\end{prop} 
Above Proposition \ref{nonkahler null} is closely related to the differentiability of the volume function $\alpha\mapsto\langle\alpha^n\rangle$. The reader can refer to \cite{Nystr19} and \cite{Vu23} for the differentiability of the volume function on the cone of big classes.

\subsection{solvability of Monge-Ampere equations in big classes and regularity results}
In this section, we note the important results from \cite{BEGZ}. 
Let $X$ be a compact K\"{a}hler manifold and $\alpha$ be a big class on $X$. Then the Monge-Amp\`{e}re equation in $\alpha$ is the equation of positive measures stated as follows:
$$
\langle T^n\rangle =\mu
$$
where $\mu$ is a given positive measure on $X$ with $\mu(X)=\langle\alpha^n\rangle$ and $T$ is a closed positive $(1,1)$-current in $\alpha$. In the context of complex geometry, the importance of the Monge-Amp\`{e}re equation comes from the fact that the equation of the K\"{a}hler-Einstein metrics $\Ric(\omega)=\lambda\omega$ is equivalent to a suitable type of Monge-Amp\`{e}re equations. \\
Boucksom-Eyssidieux-Guedji-Zeriahi \cite{BEGZ} proved the following solvability theorem in the general setting.
\begin{theo}[\cite{BEGZ} Theorem 3.1, Theorem 4.1]
Let $(X,\omega_0)$ be a compact K\"{a}hler manifold and $\alpha$ be a big class on $X$. If $\mu$ is a positive measure on $X$ that puts no mass on any pluripolar subsets and satisfies $\mu(X)=\langle\alpha^n\rangle$, then there exists a closed positive $(1,1)$-current $T$ in $\alpha$ such that
$$
\langle T^n\rangle =\mu.
$$
If $\mu=f\omega_0^n$ with $f\in L^p(\omega_0^n)$ for some $p>1$, then the solution $T$ has minimal singularities.
\end{theo}
They also proved the following regularity result if $\alpha$ is nef and big which is used in this paper:
\begin{theo}[\cite{BEGZ}]\label{monge ampere equation in nef and big}
Let $X$ be a compact K\"{a}hler manifold and $\alpha$ be a nef and big class on $X$. Let $\mu$ is a smooth volume form on $X$ with $\mu(X)=\alpha^n$ and $T$ be a closed positive $(1,1)$-current in $\alpha$ which satisfies $\langle T^n\rangle=\mu$. Then $T$ has minimal singularities and $T$ is smooth K\"{a}hler on $\Amp(\alpha)$.
\end{theo}
In the proof of Theorem \ref{monge ampere equation in nef and big} in \cite{BEGZ}, they also proved the existence of approximation sequence by K\"{a}hler metrics and the detailed positivity result of the solution $T$. We will refer to these results in Lemma \ref{positive current nef big}.

\section{Slope stability and the admissible HYM connections}
We say that a torsion free subsheaf $\mathcal{F}$ of a torsion free sheaf $\mathcal{E}$ on a complex manifold is {\it{saturated}} if the quotient $\mathcal{E}/\mathcal{F}$ is torsion free.
We define the notion of slope stability with respect to nef and big classes.
\begin{defi}\label{stability def}
Let $X$ be a compact K\"{a}hler manifold and $\alpha$ be a nef and big class on $X$. Then we say a torsion free sheaf $\mathcal{E}$ is $\alpha^{n-1}$-slope stable if, for any saturated subsheaf $\mathcal{F}$ of $\mathcal{E}$ satisfies the inequality $\mu_{\alpha}(\mathcal{F})<\mu_{\alpha}(\mathcal{E})$. Here $\mu_{\alpha}(\mathcal{F})$ is a real number called as $\alpha^{n-1}$-slope defined by
$$
\mu_{\alpha}(\mathcal{F}):=\frac{1}{\rk \mathcal{F}}\int_Xc_1(\det\mathcal{F})\wedge\frac{\alpha^{n-1}}{(n-1)!}.
$$
\end{defi}
\noindent Here $\det\mathcal{F}$ is the determinant line bundle of $\mathcal{F}$ which is isomorphic to $(\bigwedge^{\rk\mathcal{F}}\mathcal{F})^{**}$ and $c_1(\det\mathcal{F})$ is its first Chern class (refer to \cite{Kob} for more general definition).  
\begin{defi}
Let $X$ be a compact K\"{a}hler manifold and $\alpha$ be a nef and big class on $X$. Then a torsion free sheaf $\E$ is {\it{$\alpha^{n-1}$-slope polystable}} if there exists $\alpha^{n-1}$-slope stable torsion free sheaves $\E_1,\ldots,\E_l$ such that 
$$
\E|_{\Amp(\alpha)}\simeq(\E_1\oplus\cdots\oplus\E_l)|_{\Amp(\alpha)}.
$$
\end{defi}
\begin{lemm}[\cite{Chen},\cite{Jin25}]\label{invariance}
Let $X$ be a compact K\"{a}hler manifold and $\alpha\in H^{1,1}(X,\mathbb{R})$ be a nef and big class on $X$. Let $\mathcal{E}$ be a torsion free sheaf on $X$. Then $\mathcal{E}$ is $\alpha^{n-1}$-slope stable (polystable) if and only if $\pi^*\E/\Tor$ is $\pi^*\alpha^{n-1}$-slope stable (polystable) for any blowup $\pi:\widehat{X}\to X$.
\end{lemm}
 In Lemma \ref{max slope} and Lemma \ref{approximating stability}, we will see some basic properties of slope stability.
 
\begin{rema}
We can define the notion of slope stability with respect to big, not necessarily nef, cohomology classes on compact normal analytic varieties. The readers can refer to \cite{IJZ25} for the general definitions and some results (e.g. the Bogomolov-Gieseker inequality and the Miyaoka-Yau inequality).
\end{rema}

The following definition is a generalization of the admissible Hermitian-Yang-Mills metrics which is defined by \cite{BS} for K\"{a}hler metrics:
\begin{defi}\label{admissible HYM def}
Let $X$ be a compact K\"{a}hler manifold and $T$ be a closed positive $(1,1)$-current which is smooth K\"{a}hler on a Zariski open set $\Omega$. Let $({E},h_0)$ be a complex hermitian vector bundle on $X$.\\
$(1)$ A {\it{$T$-admissible Hermitian-Yang-Mills (HYM, in short) connection}} on $(E, h_0)$ is a smooth $h_0$-connection $\nabla$ on $E|_{\Omega}$ which satisfies the followings:
\begin{itemize}
\item (integrability) $\delbar^{\nabla}\circ\delbar^{\nabla}=0$.
\item ($T$-HYM equation) $\sqrt{-1}\Lambda_TF_{\nabla}=\lambda\Id$ on $\Omega$ where $\lambda$ is a real number and
\item ($T$-admissible condition) $\int_{\Omega}|F_{\nabla}|_T^2T^n<\infty$.
\end{itemize}
\noindent We call $\lambda$ as the $T$-HYM constant of $\nabla$.\\
\noindent $(2)$ The $T$-admissible HYM metric on a holomorphic vector bundle $E$ is a smooth hermitian metric $h$ on $E|_{\Omega}$ such that the Chern connection $\nabla_h$ of $h$ is  a $T$-admissible HYM connection of $(E, h)$.
\end{defi}

\section{Existence of admissible HYM metrics}
\subsection{choice of positive currents}
To solve the $T$-HYM equation, we put the following assumption:
\begin{assu}\label{assumption2}
Let $(Y,\omega_Y)$ be a compact K\"{a}hler manifold. 
Let $\omega$ be the solution to the complex Monge-Amp\`{e}re equation $\omega^n=e^F\omega_Y^n$, where $F$ satisfies the following conditions:
\begin{itemize}
\item there is an analytic subset $Z\subset Y$ with $\codim Z\ge 2$. For simplicity, we assume that $Z$ is irreducible and its codimension equals to $l\ge 2$.
\item $F\in C^{\infty}(Y\setminus Z)$ and there is a constant $0\le a \le 1/2$ such that $F=-a\log(|f_1|^2+\cdots+|f_l|^2)+O(1)$ locally where $(f_1,\ldots,f_l)$ is a local defining functions of $Z$.
\item there exists a modification $\pi:X\to Y$ along $Z$ and a sequence of K\"{a}hler metrics $\omega_i$ in $\pi^*\{\omega\}+(1/i)\omega_X$ such that 
\begin{itemize}
\item[$\bullet$] $\omega_i\to \pi^*\omega$ in $C^{\infty}_{\loc}(X\setminus D)$ and weakly on $X$ where $D=\pi^{-1}(Z)$ is the exceptional divisor.
\item[$\bullet$] $\int_X\left|\log\left(\frac{\omega_i^n}{\omega_X^n}\right)\right|^p\omega_i^n\le C$ for some $p>n$. 
\end{itemize}
\end{itemize}
We furthermore assume that
\begin{itemize}
\item $(E,h_0)=(\pi^*F,\pi^*h_F)$ where $\pi:X\to Y$ is a modification as above, $F$ is a holomorphic vector bundle on $Y$ and $h_F$ is a smooth hermitian metric on $F$.
\item $\alpha:=\pi^*\{\omega\}$ and $E$ is $\alpha^{n-1}$-slope stable.
\item We denote by $T:=\pi^*\omega$ where $\pi:X\to Y$ is a modification as above.
\end{itemize}
\end{assu}

\begin{lemm}[\cite{BEGZ}, \cite{CCHTST25}]\label{regularity of omega}
Let $\omega$ be a solution to $\omega^n=e^F\omega_Y^n$ in Assumption \ref{assumption2}. Then the following holds:\\
$(1)$ {\rm{\cite{BEGZ}}} The solution $\omega$ exists and it is smooth on $Y\setminus Z$. \\
$(2)$ {\rm{\cite{CCHTST25}}} The solutoin $\omega$ is a K\"{a}hler current on $Y$. Furthermore, there exists a modification $\pi:X\to Y$ along $Z$ so that $D=\pi^{-1}(Z)$ is a snc divisor and K\"{a}hler metrics $\omega_i$ in $\pi^*\{\omega\}+(1/i)\omega_X$ such that 
\begin{itemize}
\item[{\rm{(a)}}] $(\pi^*\omega)^n=e^f|s_D|^q\omega_X^n$ for some $q\ge 1$ where $s_D$ is a canonical section of $D$ and $f\in C^{\infty}(X)$.
\item[{\rm{(b)}}] $\omega_i\to \pi^*\omega$ in $C^{\infty}_{\loc}(X\setminus D)$ and weakly on $X$ where $D=\pi^{-1}(Z)$ is the exceptional divisor.
\item[{\rm{(c)}}] There exists a constant $C>0$ such that $\int_X\left|\log\left(\frac{\omega_i^n}{\omega_X^n}\right)\right|^p\omega_i^n\le C$.
\end{itemize} 
\end{lemm}
\begin{proof}
$(1)$ The solution $\omega$ exists by Theorem 3.1 in \cite{BEGZ} and $\omega$ has bounded potentials by Theorem 4.1 in \cite{BEGZ}.
Let $\pi:X\to Y$ be a sequence of blowups of $Y$ along $Z$ such that $D=\pi^{-1}(Z)$ is a snc divisor. Let $s_D$ be a global section of a holomorphic line bundle $\mathcal{O}(D)$ so that ${\rm{div}}(s_D)=D$. Since $\pi^*\omega_Y^n=|s_D|^{2(l-1)}\omega_X^n$,  we see that $\pi^*\omega$ is a solution to the complex Monge-Amp\`{e}re equation $(\pi^*\omega)^n=e^f|s_D|^{2l-2-p}\omega_X^n$ where $f\in C^{\infty}(X)$. Since $l\ge 2$ and $1\ge p\ge 0$, we have $2l-2-p\ge 1$. Since the cohomology class $\{\pi^*\omega\}$ is nef and big, Theorem 5.1 in \cite{BEGZ} (refer to Lemma \ref{positive current nef big} below) shows that $\pi^*\omega$ is smooth on $X\setminus D$, since the exceptional divisor $D$ coincides with the nonK\"{a}hler locus of $\{\pi^*\omega\}$.\\
$(2)$ Since $F$ is smooth on $Y\setminus Z$, $e^F\in L^p(\omega_Y^n)$ for some $p>1$ and ${\rm{Ric}}(\omega)\ge -C\omega_Y$, the second assertion is shown by \cite{CCHTST25}. 
\end{proof}

\noindent All proofs in this paper hold for $T=\pi^*\omega$ in Assumption \ref{assumption2}. \\
We also consider the following conditions for currents:
\begin{assu}\label{assumption3}
Let $(X,\omega_0)$ be a compact K\"{a}hler manifold and $\alpha$ be a nef and big class such that $D=E_{nK}(\alpha)$ is a snc divisor. Let $T$ be a closed positive $(1,1)$-current in $\alpha$ satisfying $\langle T^n\rangle=e^f|s_D|^{2a}\omega_0^n$ where $a\ge 0$. 
\end{assu}
\begin{lemm}[\cite{BEGZ}]\label{positive current nef big}
A closed positive $(1,1)$-current $T$ in Assumption \ref{assumption3} always exists and it satisfies the following conditions:
\begin{enumerate}
\item $T$ is smooth K\"{a}hler on $X\setminus D$ and $T\ge |s_D|^{2m}\omega_0$ for some constant $m\ge0$, where $s_D$ is a canonical section of $D$.
\item There exists a sequence of K\"{a}hler metrics $\omega_i$ in $\alpha+(1/i)\omega_0$ such that
\begin{itemize}
\item $\omega_i\to T$ in $C^{\infty}_{\loc}(X\setminus D)$ and weakly on $X$,
\item $\omega_i^n=e^{F_i}\omega_0^n$ for some $F_i\in C^{\infty}(X)$ such that $\int_X|F_i|^pe^{F_i}\omega_0^n\le C$ for any $i$.
\end{itemize}
\end{enumerate}
\end{lemm}
\begin{proof}
Since the result follows from the proof of \cite{BEGZ} Theorem 5.1, we only sketch the proof.
Let $T$ be a closed positive $(1,1)$-current in $\alpha$ with minimal singularities which satisfies the complex  Monge-Amp\`{e}re equation
$$
\langle T^n\rangle=e^F\omega_0^n
$$
as in the statement. By \cite{BEGZ}, such $T$ exists and it is smooth K\"{a}hler on  $X\setminus D$.  We denote by $T=\theta+dd^c\varphi$. Let us define a smooth form by $\theta_i:=\theta+(1/i)\omega_0$ in a K\"{a}hler class $\alpha+(1/i)\omega_0$ and consider a smooth K\"{a}hler metric $\omega_i=\theta_i+dd^c\varphi_i$ defined by
$$
\omega_i^n=c_i(e^F+\frac{1}{i})\omega_0^n.
$$ 
Such $\omega_i$ exists by \cite{Yau78}.
We also consider $\varphi_{D'}$ defined by $[D']=\theta-\omega_0+dd^c\varphi_{D'}$. Then $\varphi_{D'}\ge \log|s_D|^{2k}-C$ for some $k>0$.
Here we remark that $F$ is upper bounded.
Then, by the proof of \cite{BEGZ} Theorem 5.1, we have
$$
\sqrt{-1}\Lambda_{\omega_0}dd^c\varphi_i=\Delta_{\omega_0}\varphi_i\le C_1\exp(C(C_2-\varphi_{D'}))\le C|s_D|^{-2k}
$$
for any $i>0$. Since $\theta_i\le \theta_1$ is smooth, we obtain $\sqrt{-1}\Lambda_{\omega_0}\omega_i\le C|s_D|^{-2k}$. Let us consider $\omega_0=\Sigma_{k=1}^n\sqrt{-1}dz_k\wedge d\overline{z_k}$ and $\omega_i=\Sigma_{k=1}^n\sqrt{-1}g_{i,k}dz_k\wedge d\overline{z_k}$ at a point $p$. Then we have $g_{i,k}\le C|s_D|^{-2k}$. Since $\omega_i$ satisfies the Monge-Amp\`{e}re equation, we have $g_{i,k}\ge C|s_D|^{2m}$ for some $m>0$. Hence we have $\omega_i\ge C|s_D|^{2m}\omega_0$ for all $i$. 

In \cite{BEGZ}, they showed that $\varphi_i\to \varphi$ in $L^1$-topology. Let us fix an Euclidean ball $B$ in a coordinate neighborhood in ${X}\setminus D$. Then, \cite{BEGZ} showed that there exists a constant $C_B>0$ such that the uniform estimate $\|\Delta_{\omega_0}\varphi_i\|_{L^{\infty}(B)}\le C_B$ holds for any $i$. It is also shown by \cite{GPTW24} that $\|\varphi_i\|_{L^{\infty}(B)}\le C_B$. Then, by the standard elliptic estimate (e.g. \cite{GilT} Theorem 9.11), we obtain the uniform $L^{p}_2$-estimate on $B$, that is, we obtain $\|\varphi_i\|_{L^p_2(B)}\le C_B$ for any $i$ and $p>0$. Then, the proof of \cite{Tru84} leads to the uniform $C^{2,\alpha}(B)$ estimate $\|\varphi_i\|_{C^{2,\alpha}(B)}\le C_B$. Then, by the classical Schauder estimates (e.g. \cite{GilT} Problem 6.1), we obtain the uniform $C^{k,\alpha}(B)$ estimate, that is, we have $\|\varphi_i\|_{C^{k,\alpha}(B)}\le C_B$ for any $k$. Therefore $\varphi_i$ converges to $\varphi$ locally smoothly on ${X}\setminus D$.
\end{proof}
\noindent If we can prove (\ref{L-infty estimate laplacian}) in Proposition \ref{L-infty estimate} for $T$ in Assumption \ref{assumption3}, then all remaining results in this paper valid for such $T$. 
We will use the following result.
\begin{lemm}[\cite{GPSS23} Theorem 2.1]\label{sob ineq}
Let $\omega_i$ be a sequence of K\"{a}hler metrics in Assumption \ref{assumption2} or Lemma \ref{positive current nef big}. Then there exist constants $q>1$ and $C>0$ which are independent of $i$ such that for any $\omega_i$ and any $u\in L^2_1(X)_{\omega_i}$, we have the following Sobolev-type inequality
$$
\left(\frac{1}{V_{\omega_i}}\int_X|u-\overline{u}|^{2q}\omega_i^n \right)^{1/q}\le \frac{C}{V_{\omega_i}}\int_X|\nabla u|^2_{\omega_i}\omega_i^n
$$
where $\overline{u}=\frac{1}{V_{\omega_i}}\int_Xu\omega_i^n$ is the average of u over $(X,\omega_i)$.
\end{lemm}

\subsection{Proof of Theorem \ref{toward nef and big class}}
We again state Theorem \ref{toward nef and big class}.
\begin{theo}[=Corollary \ref{Uhlenbeck compactness} and Proposition \ref{nontrivial limit}]\label{limiting objects}
Let us work under Assumption \ref{assumption2}. Then there exists the following data:
\begin{itemize}
\item (Corollary \ref{Uhlenbeck compactness})  An analytic subset $\Sigma$ in $X\setminus D$ of $\codim \Sigma \ge 2$. Denote by $Z=D\cup \Sigma$.
\item  (Corollary \ref{Uhlenbeck compactness})  A  holomorphic hermitian vector bundle $E_{\infty}$ on $X\setminus Z$ whose underlying complex hermitian vector bundle coincides with the underlying complex hermitian vector bundle of $(E|_{X\setminus Z},h_0)$.
\item  (Corollary \ref{Uhlenbeck compactness}) A $T$-admissible HYM connection  $\nabla_{\infty}$ over $(E_{\infty}, h_0)$ and
\item (Proposition \ref{nontrivial limit}) nontrivial sheaf morphisms $\Phi_{\infty}:E_{\infty}\to E|_{X\setminus Z}$ and $\Psi_{\infty}: E|_{X\setminus Z}\to E_{\infty}$.
\end{itemize}
\end{theo}

The proof of Theorem \ref{limiting objects} is divided in the following two subsections.

\subsubsection{Uhlenbeck limit of the HYM connections}
Although the following lemma was proved in \cite{Jin25}, we include the proof for the completeness. The proof of the following lemma also works if $\alpha$ is merely big, not necessarily nef. 
\begin{lemm}\label{max slope}
$(1)$ Let $\eta\in H^{2n-2}(X,\mathbb{R})$ be a cohomology class which is represented by a positive $(2n-2)$-current.
Then there is a constant $C>0$ such that for any subsheaf $\mathcal{F}\subsetneq E$, the following inequality holds:
$$
\int_Xc_1(\mathcal{F})\wedge\eta\le C.
$$
\noindent $(2)$ There is a nontrivial subsheaf $\mathcal{F}_0\subsetneq E$ such that
$$
\mu_{\alpha}(\mathcal{F}_0) = \sup\{\mu_{\alpha}(\mathcal{F})\mid 0\neq \mathcal{F}\subsetneq E\}.
$$
\end{lemm}
\begin{proof}
$(1)$ The inclusion $\mathcal{F}\subset E$ induces a sheaf morphism $\iota:\det\mathcal{F}\to \bigwedge^{\rk\mathcal{F}}E$ which is injective on the locally free locus of $\mathcal{F}$ which is of codimension at least $3$. Since $\det\mathcal{F}$ is torsion free, $\iota$ is injective. Furthermore $\det\mathcal{F}$ is a subbundle of $\bigwedge^{\rk \mathcal{F}}E$ since $\det\mathcal{F}$ is a subbundle away from an analytic subset of codimension 3 and its saturation gives a subbundle.
Let $h$ be a hermitian metric on $\bigwedge^{\rk\mathcal{F}} E$. We denote by $h_{\det\mathcal{F}}$ the restriction of $h$ to $\det\mathcal{F}$. Then the curvature of $h_{\det\mathcal{F}}$ satisfies $F_{h_{\det\mathcal{F}}}=p\circ F_h \circ p + \delbar p\wedge \partial p$ where $p:\bigwedge^{\rk \mathcal{F}}E\to \bigwedge^{\rk\mathcal{F}}E$ is the $h$-orthogonal projection to $\det\mathcal{F}$ (e.g. \cite{Kob}).  Thus we can calculate as follows:
\begin{align*}
\int_Xc_1(\mathcal{F})\wedge\eta
&=\int_X\sqrt{-1}\Tr(p\circ F_h \circ p + \delbar p\wedge \partial p)\wedge\eta\\
&\le \sqrt{-1}\int_X\Tr(p\circ F_h\circ p)\wedge \eta\\
&\le\|F_h\|_{L^{\infty}}\rk(\mathcal{F})\int_X\omega_0\wedge\eta=C.
\end{align*}
$(2)$ Since we only consider the supremum, we can assume $-C< \mu_{\alpha}(\mathcal{F})$ for $\mathcal{F}\subsetneq E$. We can find elements $\eta_1,\ldots, \eta_k\in H^{2n-2}(X,\mathbb{Q})$ which forms a basis of $H^{2n-2}(X,\mathbb{Q})$ such that
\begin{itemize}
\item each $\eta_i$ is represented by a positive $(2n-2)$-current,
\item $\eta_1,\ldots, \eta_{k-1}$ lies near $\alpha^{n-1}$ and
\item there are nonnegative constants $a_1,\ldots, a_k\ge 0$ such that $\alpha^{n-1}=\Sigma_{i=1}^{k-1}a_i\eta_i-a_k\eta_k.$
\end{itemize}
These $\eta_i$ exist since $H^{2n-2}(X,\mathbb{Q})$ is dense in $H^{2n-2}(X,\mathbb{R})$ and the set $\{\alpha^{n-1}-\Sigma_{i=1}^{k-1}a_i\eta_i\mid a_i\in (1-\varepsilon,1+\varepsilon)\}$ is an open set.
By $(1)$, we have a constant $C>0$ so that $-C<\mu_{\alpha}(\mathcal{F})< C$ holds for  subsheaf $\mathcal{F}$, since $\alpha^{n-1}$ is represented by a positive current $\langle T_{\min}^{n-1}\rangle$ where $T_{\min}$ is a closed positive $(1,1)$-current with minimal singularities in $\alpha$.  Thus we also have $-C< \mu_{\eta_i}(\mathcal{F})<C$ for any $i=1,\ldots k-1$ and $\mathcal{F}\subsetneq E$ since for each $\mathcal{F}$, the slope $\mu_{\eta}(\mathcal{F})$ is a continuous function in $\eta\in H^{2n-2}(X,\mathbb{R})$. It implies the boundedness $-C<\mu_{\eta_k}(\mathcal{F})<C$. In fact, the upperboundedness follows from $(1)$. If there is no lower bound, there is a sequence of subsheaves $(\mathcal{F}_j)_j$ such that $\mu_{\eta_k}(\mathcal{F}_j)$ monotonically decreasing to $-\infty$. Then, the inequality
$-C<\mu_{\alpha}(\mathcal{F})=\Sigma_{i=1}^{k-1}a_i\mu_{\eta_i}(\mathcal{F})-a_k\mu_{\eta_k}(\mathcal{F})<C$ implies $\Sigma_{i=1}^{k-1}a_i\mu_{\eta_i}(\mathcal{F})$ monotonically decreases to $-\infty$. Since each $a_i$ is non-negative, it contradicts to the existence of a lower bound of $\mu_{\eta_i}(\mathcal{F})$ for $i=1,\ldots, k-1$. Since $\eta_i$ are $\mathbb{Q}$-cohomology classes, there is a positive integer $m$ such that $m\eta_i$ are $\mathbb{Z}$-cohomology classes. Of course $\mu_{m\eta_i}(\mathcal{F})$ is bounded. We also remark that $\int_Xc_1(\mathcal{F})\wedge m\eta_i$ is an integer. Then, we can see that
$$
\{\mu_{m\eta_i}(\mathcal{F})\mid 0\neq \mathcal{F}\subsetneq E, \mu_{\eta_i}(\mathcal{F})> -C\}\subset [-C,C]\cap \mathbb{Z}
$$
and the LHS is a finite set. Thus we obtain that
$$
\{\mu_{\alpha}(\mathcal{F})\mid 0\neq \mathcal{F}\subsetneq E, \mu_{\alpha}(\mathcal{F})> -C\}
$$
is a finite set and thus there is a nontrivial subsheaf $\mathcal{F}_0\subsetneq E$ which attains the supremum of the $\alpha^{n-1}$-slope.
\end{proof}
We remember that $E$ is $\alpha^{n-1}$-slope stable. Let us define K\"{a}hler classes $\alpha_i:=\alpha+(1/i)\omega_0$ for $i>0$.
\begin{lemm}\label{approximating stability}
$E$ is $\alpha_i^{n-1}$-slope stable for $i \gg 0$.
\end{lemm}
\begin{proof}
Let $0\neq\mathcal{F}_0\subsetneq E$ be a subsheaf with maximal slope in Lemma \ref{max slope} $(2)$. Then there is a large $i_0>0$ which is independent of subsheaves of $E$ such that for any $i\ge i_0$ and $0\neq \mathcal{F}\subsetneq E$, the following inequality holds:
\begin{align*}
&\mu_{\alpha_i}(E)-\mu_{\alpha_i}(\mathcal{F})\\
&=\mu_{\alpha}(E)-\mu_{\alpha}(\mathcal{F})+\Sigma_{j=1}^{n-1}C_j(\frac{1}{i})^j\left(\frac{\int_Xc_1(E)\wedge\alpha^{n-1-j}\wedge\omega_0^j}{\rk E}-\frac{\int_Xc_1(\mathcal{F})\wedge\alpha^{n-1-j}\wedge\omega_0^j}{\rk \mathcal{F}} \right)\\
&\ge \mu_{\alpha}(E)-\mu_{\alpha}(\mathcal{F}_0)-\frac{1}{i}C\\
&>0.
\end{align*}
Here the inequality in the third line follows from Lemma \ref{max slope} since $\alpha^{n-1-j}\wedge\omega_0^j$ is represented by a positive current $\langle T_{\min}^{n-1-j}\rangle\wedge\omega_0^j$, where $T_{\min}$ is a closed positive $(1,1)$-current with minimal singularities in $\alpha$. And the last inequality comes from the $\alpha^{n-1}$-slope stability of $E$.
\end{proof}
We denote by $\underline{E}$ the underlying complex vector bundle of $E$.
By Lemma \ref{approximating stability} and the theorem of Uhlenbeck and Yau \cite{UY}, there is the $\omega_i$-HYM metric $h_i=h_0H_i$ on $E$, where $H_i\in \Herm^+(\underline{E}, h_0)$ is a positive definite $h_0$-hermitian smooth endomorphism of $\underline{E}$ with $\det H_i=1$.
\begin{lemm}\label{construction of HYM conns}
$(1)$ $\nabla_i:=H_i^{\frac{1}{2}}\circ \nabla_{h_i}\circ H_i^{-\frac{1}{2}}$ is an $\omega_i$-HYM connection of a complex hermitian vector bundle $(\underline{E},h_0)$.\\
\noindent $(2)$ There is a constant $C>0$ such that $\int_X|F_{\nabla_i}|_{\omega_i,h_0}^2\omega_i^n\le C$ holds for any $i$. Here $F_{\nabla}$ denotes the curvature tensor of a connection $\nabla$.
\end{lemm}
\begin{proof}
(1) We only show that $\nabla_i$ is a $h_0$-connection.
Let $u$ and $v$ be a local section of $\underline{E}$. Then the following holds:
\begin{align*}
d\left(h_0(u,v)\right)
&=d\left(h_i(H_i^{-\frac{1}{2}}u, H_i^{-\frac{1}{2}}v)\right)\\
&=h_i\left(\nabla_{h_i}(H_i^{-\frac{1}{2}}u), H_i^{-\frac{1}{2}}v\right)+h_i\left(H_i^{-\frac{1}{2}}u,\nabla_{h_i}(H_i^{-\frac{1}{2}}v)\right)\\
&=h_0\left((H_i^{\frac{1}{2}}\circ\nabla_i\circ H_i^{-\frac{1}{2}})u, v \right)+h_0\left(u, (H_i^{\frac{1}{2}}\circ\nabla_i\circ H_i^{-\frac{1}{2}})v \right) \\
&=h_0(\nabla_iu, v)+h_0(u,\nabla_iv).
\end{align*}
 
 $(2)$ is an consequence of the well-known equation (e.g. \cite{Kob})
$$
(2c_2(E)-c_1(E)^2)\alpha_i^{n-2}=c_n\int_X(|F_{\nabla_i}|_{\omega_i,h_0}^2-|\Lambda_{\omega_i}F_{\nabla_i}|_{\omega_i,h_0}^2)\omega_i^n
$$
and the $\omega_i$-HYM condition of $\nabla_i$.
\end{proof}
Since $\omega_i$ in Assumption \ref{assumption2} smoothly converges to $T$ on $X\setminus D$ and $T$ is smooth K\"{a}hler on $X\setminus D$, we can take the Uhlenbeck limit (\cite{Nak88} and \cite{CW22}) of $\nabla_i$. That is, there are the following data:
\begin{corr}[\cite{CW22}, \cite{Nak88}]\label{Uhlenbeck compactness}
The following data exist:
\begin{itemize}
\item An analytic subset $\Sigma\subset X\setminus D$  with $\codim \Sigma \ge 2$. Denote by $Z=D\cup\Sigma$.
\item  A $T$-admissible HYM connection $\nabla_{\infty}$ on $(\underline{E}|_{X\setminus Z},h_0)$ such that
\item $\nabla_i\to \nabla_{\infty}$ in $C^{\infty}_{\loc}(X\setminus Z)$ up to $h_0$-unitary transformation.
\item  The holomorphic vector bundle $E_{\infty}:=(\underline{E}|_{X\setminus Z}, \delbar^{\nabla_{\infty}})$ given by $\nabla_{\infty}$. 
\end{itemize}
\end{corr}
From now on, we denote as $E_i:=(\underline{E}, \delbar^{\nabla_i})$ the holomorphic vector bundle given by $\nabla_i$, and $Z:=D\cup \Sigma$ a closed subset of $X$. By \cite{Bando91}, we know that
\begin{lemm}[\cite{Bando91}]
The holomorphic vector bundle $E_{\infty}$ over $X\setminus Z$ extends to the reflexive sheaf, also denoted by $E_{\infty}$, over $X\setminus D$.
\end{lemm}
\noindent The goal is to prove that $E_{\infty}$ is holomorphically isomorphic to $E|_{X\setminus D}$.
To construct the isomorphism, let us consider the following morphism:
\begin{lemm}\label{app isoms}
The section $\Phi_i:=H_i^{-\frac{1}{2}}\in H^0(X, \Hom(E_i, E))$ gives a holomorphic isomorphism from $E_i$ to $E$.
Here the holomorphic structure of $\Hom(E_i,E)$ is given by the $\omega_i$-HYM connection $\nabla_i$ on $E_i$ in Lemma \ref{construction of HYM conns} and the Chern connection $\nabla_{h_0}$ of $(E,h_0)$.
\end{lemm}
\begin{proof}
The proof is the direct consequence of the definition of $\nabla_i$ and $\delbar^{\nabla_{h_i}}=\delbar^{E}$.
\end{proof}
In the next subsection, we prove that $\Phi_i$ converges to a nontrivial holomorphic morphism $\Phi_{\infty}$ from $E_{\infty}$ to $E|_{X\setminus Z}$.

\subsubsection{nontrivial limit of $\Phi_i$}
The assumptions $h_0=\pi^*h_F$ and $T=\pi^*\omega\ge C^{-1}\pi^*\omega_Y$ are essential in the following proposition.
\begin{prop}\label{L-infty estimate}
We consider the setting in Assumption \ref{assumption2}.
Then the following holds:
\begin{itemize}
\item By replacing $\Phi_i$ by $\Phi_i/\int_X|\Phi_i|_{h_0}^2\omega_i^n$, we have $\|\Phi_{i}\|_{L^{\infty}(X)}<C$.
\item By replacing $\Phi_i^{-1}$ by $\Psi_i:=\Phi_i^{-1}/\int_X|\Phi_i^{-1}|_{h_0}^2\omega_i^n$, we have $\|\Psi_i\|_{L^{\infty}(X)}<C$.
\end{itemize}
\end{prop}
\begin{proof}
We prove the uniform $L^{\infty}(X)$-estimate on $\Phi_{i}$.
Let $\omega_i$ be a K\"{a}hler metric in $\alpha+(1/i)\omega_0$ such that $\omega_i\to T$ in ${C}^{\infty}_{\loc}(X\setminus D)$ and weakly on $X$ as in Assumption \ref{assumption2}. We give a hermitian metric $h_0^*\otimes h_0$ on the underlying smooth bundle of ${\rm{Hom}}(E_i,E)$ and $(h_0^*\otimes h_0)$-connection $\nabla_i^*\otimes\nabla_{h_0}$ which is denoted by $\nabla_i$.
Let $x\in X$ be any point. If we denote by $e_1,\ldots,e_n$ a local frame of $T^{1,0}_X$, then the following calculation works at $x$:
\begin{align}\label{L-infty estimate laplacian}
&\Box^{\delbar}_{\omega_i}\log(|\Phi_i|^2+1)\notag\\
&=i(\overline{e_i})\partial_j\left((|\Phi_i|^2+1)^{-1}\langle\Phi_i,\partial^{\nabla_i}\Phi_i\rangle\right)\notag\\
&=-(|\Phi_i|^2+1)^{-2}\langle\partial^{\nabla_i}_j\Phi_i,\Phi_i\rangle\langle\Phi_i,\partial^{\nabla_i}_j\Phi_i\rangle+(|\Phi_i|^2+1)^{-1}\left(|\partial^{\nabla_i}_j\Phi_i|^2+\langle\Phi_i,\delbar^{\nabla_i}_j\partial^{\nabla_i}_j\Phi_i\rangle\right)\notag\\
&\ge (|\Phi_i|^2+1)^{-1}\langle \Phi_i, \Lambda_{\omega_i}F_{h_0}\circ \Phi_i-\Phi_i\circ \Lambda_{\omega_i}F_{\nabla_i}\rangle\notag\\
&\ge (|\Phi_i|^2+1)^{-1}|\Phi_i|^2\left(-|\Lambda_{\omega_i}F_{h_0}|_{h_0^*\otimes h_0}-|\lambda_i|\rk E\right)\notag\\
&\ge-|\Lambda_{\omega_Y}F_{h_F}|_{h_0^*\otimes h_0}-|\lambda_{\infty}|\rk E\notag\\
&\ge -N_1.
\end{align}
In the third inequality, we used $\lambda_i\to\lambda_{\infty}$. Since the LHS is smooth on $X$, we obtain 
\begin{equation}\label{lap ineq}
\Delta_{\omega_i}\log(|\Phi_i|^2_{h_0}+1)\ge -N
\end{equation}
on $X$.
Here we denote $\Phi_i/\int_X|\Phi_i|^2\omega_i^n$ by the same notation $\Phi_i$. We set $u_i=\log(|\Phi_i|^2+1)$. Then we show that there exists a constant $C>0$ such that $\sup_Xu_i\le C$ holds for any $i$. Although the following proof is the same with \cite{GPSS23} Lemma 6.3, we include the proof for completeness. We omit to write $\omega_i^n$ in the integration of $u_i$. We replace $u_i$ by $u_i+C^2$. Let $m>0$ be a positive number. By (\ref{lap ineq}), we have
$$
C\int_Xu_i^m\ge \int_Xu_i^m\cdot(-\Delta_{\omega_i}u_i)=m\int_X\langle u_i^{m-1}\nabla u_i,\nabla u_i\rangle_{\omega_i}=m\int_Xu_i^{m-1}|\nabla u_i|^2_{\omega_i}.
$$
Hence, by $\frac{1}{2m}\le \frac{2m}{(m+1)^2}$, we obtain
\begin{equation}\label{L-infty-eq1}
\frac{1}{m}\int_X|\nabla(u_i^{\frac{m+1}{2}})|^2_{\omega_i}\le {C}\int_Xu_i^m.
\end{equation}
By Theorem \ref{sob ineq}, we have
\begin{align*}
\left(\frac{C}{V_{\omega_i}}\int_X \big|\nabla(u_i^{\frac{m+1}{2}})\big|^2_{\omega_i}\right)^{1/2}
&\ge \left(\frac{1}{V_{\omega_i}}\int_X \big|u_i^{\frac{m+1}{2}}-\overline{u_i^{\frac{m+1}{2}}}\big|^{2q} \right)^{1/(2q)}\\
&\ge \left(\frac{1}{V_{\omega_i}}\int_X u_i^{(m+1)q}\right)^{1/(2q)}
 -\left(\frac{1}{V_{\omega_i}}\int_X u_i^{m+1}\right)^{1/2}.
\end{align*}
and thus,
\begin{equation}\label{L-infty-eq2}
\left(\frac{1}{V_{\omega_i}}\int_X u_i^{(m+1)q}\right)^{1/q}
\le \frac{C}{V_{\omega_i}}\int_X \big|\nabla(u_i^{\frac{m+1}{2}})\big|^2_{\omega_i}
 +\frac{1}{V_{\omega_i}}\int_X u_i^{m+1}.
\end{equation}
By (\ref{L-infty-eq1}) and (\ref{L-infty-eq2}), we get
\begin{equation}\label{L-infty-eq3}
\left(\frac{1}{V_{\omega_i}}\int_Xu_i^{(m+1)q}\right)^{1/q}\le \frac{1}{V_{\omega_i}}\left(C^2m\int_Xu_i^m+\int_Xu_i^{m+1}\right).
\end{equation}
Since $u_i\ge C^2$, the following calculation works. Here $N>\frac{q}{q-1}$ and $N^*=\frac{N}{N-1}$.
\begin{align}\label{L-infty-eq4}
\frac{C^2}{V_{\omega_i}}\int_Xu_i^m
&\le \frac{C^2}{V_{\omega_i}}\int_X\frac{u_i^{m+1}}{C^2}\notag\\
&\le \left(\frac{1}{V_{\omega_i}}\int_Xu_i^{(m+1)N^*}\right)^{1/N^*}\notag\\
&\le \left(\frac{1}{V_{\omega_i}}\int_Xu_i^{(m+1)\frac{q}{(q-1)(N-1)}}\cdot u_i^{(m+1)\frac{N(q-1)-q}{(q-1)(N-1)}} \right)^{1/N^*}\notag\\
&\le \left(\frac{1}{V_{\omega_i}}\int_Xu_i^{(m+1)q}\right)^{\frac{1}{(q-1)N}}\left(\frac{1}{V_{\omega_i}}\int_Xu_i^{m+1}\right)^{\frac{N(q-1)-q}{(q-1)N}}\notag\\
&\le \frac{1}{2m}\|u_i^{m+1}\|_{L^q,\omega_i}+(2m)^{\frac{q}{(q-1)N-q}}\|u_i^{m+1}\|_{L^1,\omega_i}.
\end{align}
In the last inequality, we use the Young's inequality. By (\ref{L-infty-eq3}) and (\ref{L-infty-eq4}), we get
\begin{equation}\label{L-infty eq5}
\left(\frac{1}{V_{\omega_i}}\int_Xu_i^{(m+1)q}\right)^{1/q}\le \frac{C}{V_{\omega_i}}((2m)^{\alpha}+1)\int_Xu_i^{m+1}
\end{equation}
where $\alpha=\frac{(q-1)N}{(q-1)N-q}$. Hence we obtain 
$$
\left(\frac{1}{V_{\omega_i}}\int_Xu_i^{(m+1)q}\right)^{\frac{1}{q(m+1)}}
\le C^{\frac{1}{m+1}}(2m+1)^{\frac{\alpha}{m+1}}\left(\frac{1}{V_{\omega_i}}\int_Xu_i^{m+1}\right)^{\frac{1}{m+1}}.
$$
If we set $m=q^j$ for $j\in \mathbb{Z}_{\ge0}$ and repeatedly use the above inequality, we obtain
\begin{equation}\label{L-infty-eq6}
\left(\frac{1}{V_{\omega_i}}\int_Xu_i^{(q^j+1)q}\right)\le C^{\Sigma_{j=0}^{\infty}\frac{1}{q^j+1}}\prod_{j=0}^{\infty}(q^j+1)^{\frac{\alpha}{q^j+1}}\left(\frac{1}{V_{\omega_i}}\int_Xu_i^2\right)^{\frac{1}{2}}
\end{equation}
for any $j$.
Therefore, by $j\to \infty$, we obtain 
$$
\sup_Xu_i\le C\left(\frac{1}{V_{\omega_i}}\int_Xu_i^2\right)^{\frac{1}{2}}.
$$
We now recall that $u_i=\log(|\Phi_i|^2+1)+C^2$, $\left(\log(|\Phi_i|^2+1)\right)^2\le |\Phi_i|^2$ and $\int_X|\Phi_i|^2\omega_i^n=1$. Thus we get
$$
\sup_X\log(|\Phi_i|^2+1)\le C
$$
and it shows the result.
We can prove the second assertion by replacing $\nabla_i:=\nabla_i^*\otimes\nabla_{h_0}$ by $\widetilde{\nabla_i}:=\nabla_{h_0}^*\otimes\nabla_i$.
\end{proof}
\noindent From now on, we denote the normalized morphism $\Phi_i/(\int_X|\Phi_i|_{h_0}^2\omega_i^n)$ by $\Phi_i$ and\\ $\Phi_i^{-1}/\int_X|\Phi_i^{-1}|_{h_0}^2\omega_i^n$ by $\Psi_i$. Furthermore,
\begin{assu}\label{assumption4}
We work under Assumption \ref{assumption3} in the rest of this paper by assuming that Proposition \ref{L-infty estimate} holds for a closed positive $(1,1)$-current $T$ in a nef and big class $\alpha$ in Assumption \ref{assumption3}.
\end{assu}
Then we can derive the existence of a nontrivial limit of $\Phi_i$ and $\Psi_i$.
\begin{prop}\label{nontrivial limit}
$(1)$ There exists a subsequence $(\Phi_{i_j})_j$ of $(\Phi_i)_i$ and a nonzero holomorphic global section $\Phi_{\infty}\in H^0(X\setminus Z, \Hom(E_{\infty}, E))$ such that $\Phi_{i_j}\to \Phi_{\infty}$ in $j\to \infty$  in $L^2_{\loc}(X\setminus Z)$.\\
$(2)$ There exists a subsequence $(\Psi_{i_j})_j$ of $(\Psi_i)_i$ and a nonzero holomorphic global section $\Psi_{\infty}\in H^0(X\setminus Z, \Hom(E, E_{\infty}))$ such that $\Psi_{i_j}\to \Psi_{\infty}$ in $j\to \infty$  in $L^2_{\loc}(X\setminus Z)$.
\end{prop}
\begin{proof}
We only prove $(1)$ since the proof of $(2)$ is the same. Let $h_D'$ be a smooth hermitian metric on $\mathcal{O}(D)$. We denote by $\nabla_{i}:=\nabla_{i}^*\otimes\nabla_{h_0}\otimes\nabla_{h_D'}$ an $(h_0^*\otimes h_0\otimes h_D')$-connection on $\Hom(E_i,E)\otimes \mathcal{O}(D)$.
If we use the identity $\Box^{\partial^{\nabla_i}}_{\omega_i}=\Box^{\delbar^{\nabla_i}}_{\omega_i}-\sqrt{-1}[\Lambda_{\omega_i},F_{\nabla_i}]$ and $\delbar^{\nabla_i}(s_D^m\Phi_i)=0$, we obtain
\begin{align*}
\Delta^{\nabla_i}_{\omega_i}(s_D^m\Phi_i)
&=2\Box^{\delbar^{\nabla_i}}_{\omega_i}(s_D^m\Phi_i) -\sqrt{-1}\Lambda_{\omega_i}F_{\nabla_i}\cdot (s_D^m\Phi_i)\\
&=-\sqrt{-1}\Lambda_{\omega_i}F_{\nabla_i}\cdot (s_D^m\Phi_i).
\end{align*}
Thus we obtain
\begin{equation}\label{laplacian L-infty estimate}
\|\Delta_{\omega_i}^{\nabla_i}(s_D^m\Phi_i)\|_{L^{2}(X\setminus Z^{\varepsilon})}\le C_{\varepsilon}
\end{equation}
for any $\varepsilon>0$. Let $(B_k, D_j)_{k=1,\ldots,l, j=1,\ldots m}$ be an open cover of $X$ consisting of finite number of Euclidean open balls in $X$ such that $X\setminus Z^{\varepsilon}\Subset \bigcup_kB_k\Subset X\setminus Z^{\varepsilon/2}$, $Z^{\varepsilon/2}\Subset \bigcup_jD_j\Subset Z^{\varepsilon}$ and $(\eta_k, \psi_j)_{k,j}$ be a partition of unity associated to $(B_k, D_j)_{k,j}$. Then $\Supp(\eta_k)\subset X\setminus Z^{\varepsilon/2}$. Next we show 
\begin{equation}\label{L21 estimate}
\|\nabla_i(\eta_ks_D^m\Phi_i)\|^2_{L^2(X)}\le (C'_{\varepsilon/2}+1)\|s_D^m\Phi_i\|_{L^2(X\setminus Z^{\varepsilon/2})}^2 + \|\Delta_{\omega_i}^{\nabla_i}(s_D^m\Phi_i)\|_{L^2(X\setminus Z^{\varepsilon/2})}^2\le C_{\varepsilon}.
\end{equation}
First we compute as follows:
\begin{align*}
(\Delta^{\nabla_i}_{\omega_i}(s_D^m\Phi_i),\eta_k^2s_D^m\Phi_i)_{L^2}
&=(\nabla_i(\eta_k^2s_D^m\Phi_i),\nabla_is_D^m\Phi_i)\\
&=(2\eta_k\nabla\eta_k\cdot s_D^m\Phi_i+\eta_k^2\nabla_i(s_D^m\Phi_i),\nabla_i(s_D^m\Phi_i))\\
&=2(s_D^m\Phi_i\nabla\eta_k,\eta_k\nabla_i(s_D^m\Phi_i))+\|\eta_k\nabla_i(s_D^m\Phi_i)\|^2\\
&\ge-2(2\|s_D^m\Phi_i\nabla\eta_k\|)(\frac{1}{2}\|\eta_k\nabla_i(s_D^m\Phi_i)\|) +\|\eta_k\nabla_i(s_D^m\Phi_i)\|^2\\
&\ge-2(4\|s_D^m\Phi_i\nabla\eta_k\|^2+\frac{1}{4}\|\eta_k\nabla_i(s_D^m\Phi_i)\|^2)+\|\eta_k\nabla_i(s_D^m\Phi_i)\|^2\\
&=-8\|s_D^m\Phi_i\nabla\eta_k\|^2+\frac{1}{2}\|\eta_k\nabla_i(s_D^m\Phi_i)\|^2.
\end{align*}
Thus, we obtain
\begin{align*}
\|\eta_k\nabla_i(s_D^m\Phi_i)\|^2
&\le 2(\Delta^{\nabla_i}_{\omega_i}(s_D^m\Phi_i),\eta_k^2s_D^m\Phi_i)+16\|s_D^m\Phi_i\nabla\eta_k\|^2\\
&\le 2\|\Delta^{\nabla_i}_{\omega_i}(s_D^m\Phi_i)\|_{L^2(X\setminus Z^{\varepsilon/2})}\|s_D^m\Phi_i\|_{L^2(X\setminus Z^{\varepsilon/2})}+C_{\varepsilon/2}\|s_D^m\Phi_i\|_{L^2(X\setminus Z^{\varepsilon/2})}^2\\
&\le (1+C_{\varepsilon/2})\|s_D^m\Phi_i\|_{L^2(X\setminus Z^{\varepsilon/2})}^2 + \|\Delta^{\nabla_i}_{\omega_i}(s_D^m\Phi_i)\|^2_{L^2(X\setminus Z^{\varepsilon/2})}.
\end{align*}
Therefore,  by proposition \ref{L-infty estimate} together with $(\ref{laplacian L-infty estimate})$, the following estimate follows:
\begin{align*}
\|\nabla_i(\eta_ks_D^m\Phi_i)\|^2
&\le 2\|\nabla\eta_k\|^2\|s_D^m\Phi_i\|^2_{L^2(X\setminus Z^{\varepsilon/2})}+2\|\eta_k\nabla_i(s_D^m\Phi_i)\|^2\\
&\le 2(1+C'_{\varepsilon/2})\|s_D^m\Phi_i\|^2_{L^2(X\setminus Z^{\varepsilon/2})}+2\|\Delta^{\nabla_i}_{\omega_i}(s_D^m\Phi_i)\|^2_{L^2(X\setminus Z^{\varepsilon/2})}\\
&\le C''_{\varepsilon/2},
\end{align*}
which shows (\ref{L21 estimate}). 
By Corollary \ref{Uhlenbeck compactness}, a sequence of $\End(\underline{E})$-valued 1-forms $\nabla_i-\nabla_{\infty}=: A_i$ converges to $0$ in $C^{\infty}_{\loc}(X\setminus Z)$ up to $h_0$-unitary transformations. Remember that $\omega_i\to T$ smoothly on $X\setminus Z^{\varepsilon/2}$ and $T$ is smooth K\"{a}hler on $X\setminus Z$.  Then, since $h_0$-unitary transformations preserve the norms, we obtain
\begin{equation}\label{uniform L21 estimate}
\|\nabla_{\infty}(\eta_ks_D^m\Phi_i)\|^2_{L^2(X)}
\le C_{\varepsilon/2}\|s_D^m\Phi_i\|^2_{L^2(X\setminus Z^{\varepsilon/2})}+2\|\Delta^{\nabla_i}_{\omega_i}(s_D^m\Phi_i)\|^2_{L^2(X\setminus Z^{\varepsilon/2})} \le C'_{\varepsilon/2}. 
\end{equation}
Since $|s_D|>C_{\varepsilon}^{-1}$ on $X\setminus Z^{\varepsilon}$, we have $\int_{X\setminus Z^{\varepsilon}}|\Phi_i|^2\omega_i^n\le C_{\varepsilon}$ and $\int_{X\setminus Z^{\varepsilon}}|\nabla_{\infty}(\eta_k\Phi_i)|^2\omega_i^n\le C_{\varepsilon}$.

Next we construct a nonzero subsequential $L^2_{\loc}$-limit $\Phi_{\infty}$ of $\Phi_i$.
Since $\Supp(\eta_k)\subset X\setminus Z^{\varepsilon/2}$, each $\eta_k\Phi_i$ is compactly supported in $X\setminus Z^{\varepsilon/2}$.
We recall that the inclusion $L^2_{1,0}(X\setminus Z^{\varepsilon/2})\to L^2(X\setminus Z^{\varepsilon/2})$ is compact by the Sobolev embedding theorem (e.g. \cite{GilT} Corollary 7.11). Thus, there exists a subsequence $(\Phi_{i,\varepsilon})_{i}$ of $(\Phi_i)_i$ such that, $\Sigma_{k=1}^{l}\eta_k\Phi_{i,\varepsilon}\to \Phi_{\varepsilon}$ in $L^2(X\setminus Z^{\varepsilon/2})$ as a section of $\Hom(\underline{E},\underline{E})$. We can see $\Sigma_{k}^{l}\eta_k=1$ on $X\setminus Z^{\varepsilon}$ by the construction of $\eta_k$. Therefore we have $\Phi_{i\varepsilon}\to \Phi_{\varepsilon}$ in $L^2(X\setminus Z^{\varepsilon})$. Let $(\varepsilon_j)_j$ be a sequence of small numbers which monotonically decrease to 0. Then we can inductively extract a subsequence $(\Phi_{k,\varepsilon_j})_k$ for each $\varepsilon_j$ such that $(\Phi_{k,\varepsilon_j})_k$ is a subsequence of $(\Phi_{k,\varepsilon_{j-1}})_k$ and each $(\Phi_{k,\varepsilon_j})_k$ converges to $\Phi_{\varepsilon_j}$ in $L^2(X\setminus Z^{\varepsilon_j})$. Then the subsequence $(\Phi_{j\varepsilon_j})_j$ converges to $\Phi_{\infty}$ in $L^2_{\loc}(X\setminus Z)$, where $\Phi_{\infty}$ is an $L^2$-local section of $\Hom(\underline{E}, \underline{E})$ over $X\setminus Z$. 

We recall that $\Phi_i\to \Phi_{\infty}$ in $L^2_{\loc}(X\setminus Z)$, $\int_X|\Phi_i|^2\omega_i^n=1$ and $\sup_X|\Phi_i|\le C$. Then, 
if we put $f_i:=|\Phi_i|$ in the following claim, we obtain that $\Phi_{\infty}\ne 0$.
\begin{claim}\label{nontrivial limit lemma}
Suppose that a sequence of functions $f_i$ on $X$ satisfies the following conditions:
\begin{itemize}
\item $\int_X|f_i|^2\omega_i^n= 1$,
\item $\sup_{X\setminus Z}|f_i|\le C$,
\item $f_i\to f$ in $L^2_{\loc}(X\setminus Z)$.
\end{itemize}
Then $f\in L^2_{\loc}(X\setminus Z)$ is nonzero.
\end{claim}
\begin{proof}
By the first assumption, $|f_i|^2\omega_i^n$ subsequentially converges to a probability measure $d\mu$ weakly: $|f_i|^2\omega_i^n\to d\mu$. Hence, if $f_i\to f=0$ in $L^2_{\loc}(X\setminus Z)$, then we obtain 
$$
\int_Zd\mu
=\lim_{\varepsilon\to0}\int_{Z^{\varepsilon}}d\mu
=\lim_{\varepsilon\to 0}\lim_{i\to\infty}\int_{Z^{\varepsilon}}|f_i|^2\omega_i^n
=\int_Xd\mu-\lim_{\varepsilon\to 0}\lim_{i\to\infty}\int_{X\setminus Z^{\varepsilon}}|f_i|^2\omega_i^n
=1
$$
On the other hand, the second assumption leads to
$$
\int_Zd\mu
=\lim_{\varepsilon\to0}\int_{Z^{\varepsilon}}d\mu
=\lim_{\varepsilon\to 0}\lim_{i\to\infty}\int_{Z^{\varepsilon}}|f_i|^2\omega_i^n
\le C\lim_{\varepsilon\to 0}\lim_{i\to\infty}\int_{Z^{\varepsilon}}\omega_i^n
=C\lim_{\varepsilon\to0}\int_{Z^{\varepsilon}}\langle T^n\rangle
=0,
$$
it is a contradiction. Thus $f\ne 0$.
\end{proof}
Finally we check that $\Phi_{\infty}$ is a holomorphic section of $\Hom(E_{\infty},E)$ with respect to a holomorphic structure $\delbar_{\infty}:=(\delbar^{\nabla_{\infty}})^*\otimes \delbar^{\nabla_{h_0}}$. We denote by $\delbar_i:=(\delbar^{\nabla_i})^*\otimes \delbar^{\nabla_{h_0}}$. We remark that these operators are differential operators on a fixed smooth bundle $\Hom(\underline{E},\underline{E})$ and $\delbar_i$ smoothly converges to $\delbar_{\infty}$ on $X\setminus Z$. Then we can see $\delbar_{\infty}\Phi_{\infty}=\lim_{j\to\infty}\delbar_j\Phi_{j\varepsilon_j}=0$, where the limit means the weak convergence. Let $U\Subset X\setminus Z$ be a small open set so that $\Hom(E_{\infty}, E)$ is holomorphically trivialized. Let $V\Subset U$ be an open subset and $\eta:X\to \mathbb{R}_{\ge 0}$ be a cut-off function such that $\eta\equiv 1$ on $V$ and $\Supp(\eta)\subset U$. Then we can see that $\eta\Phi_{\infty}$ is a bounded function over $\mathbb{C}^n$ which is weakly holomorphic over $V$. Let $\rho_{\varepsilon}$ be a smoothing kernel on $\mathbb{C}^n$. Then the convolution $\rho_{\varepsilon}*(\eta\Phi_{\infty})$ is smooth on $\mathbb{C}^n$ and holomorphic on $V$. By the mean value inequality, we can see
 $$
 \|\rho_{\varepsilon}*(\eta\Phi_{\infty})\|_{L^{\infty}(B)}\le C\|\rho_{\varepsilon}*(\eta\Phi_{\infty})\|_{L^2(B)}
 $$
 for any Euclidean ball $B\Subset V$.
 Since $\rho_{\varepsilon}*(\eta\Phi_{\infty})$ converges to $\eta\Phi_{\infty}$ in the $L^2$-topology, we obtain the uniform estimate
 $$
  \|\rho_{\varepsilon}*(\eta\Phi_{\infty})\|_{L^{\infty}(B)}\le C_B.
 $$
 Since each $\rho_{\varepsilon}*(\eta\Phi_{\infty})$ is holomorphic on $B$, the Montel theorem in the theory of several complex variables implies the limit $\eta\Phi_{\infty}=\Phi_{\infty}$ is holomorphic over $B$. Thus we obtain that $\Phi_{\infty}$ is a nonzero holomorphic section of $\Hom(E_{\infty}, E)$ over $X\setminus Z$.
\end{proof}
Since $\Phi_i(z)\to \Phi_{\infty}(z)$, $\Psi_i(z)\to \Psi_{\infty}(z)$ almost everywhere $z\in X\setminus Z$, $\nabla_i\to\nabla_{\infty}$ up to $h_0$-unitary transformations and $h_0$-unitary transformations preserve norms, we obtain the following by Proposition \ref{L-infty estimate}:
\begin{corr}\label{L-infty estimate of limit object}
We obtain nontrivial sheaf morphisms $\Phi_{\infty}:E_{\infty}\to E|_{X\setminus D}$ and $\Psi_{\infty}:E|_{X\setminus D}\to E_{\infty}$ which satisfy
$\|\Phi_{\infty}\|_{L^{\infty}(X\setminus Z)}<\infty$ and  $\|\Psi_{\infty}\|_{L^{\infty}(X\setminus Z)}<\infty$
\end{corr} 
We also need the uniform estimates of the first derivatives.
\begin{prop}\label{uniform esti of diff}
There exists a constant $C>0$ such that the following inequality holds for any $i$:
$$
\int_X|(\nabla_i\otimes\nabla_{h_0}^*\otimes\nabla_{h_D^m})s_D^{2m}\Psi_i|^2_{h_0\otimes h_0^*\otimes h_{D^m}, \omega_i}\omega_i^n\le C.
$$
The same estimate holds for $\Phi_i$. In particular, we have 
$$
\int_{X\setminus Z}|(\nabla_{\infty}\otimes\nabla_{h_0}^*\otimes\nabla_{h_D^m})s_D^{2m}\Psi_{\infty}|^2_{h_0\otimes h_0^*\otimes h_{D^m}, T}T^n<\infty.
$$
\end{prop}
\begin{proof}
Let $h_D'$ be a smooth hermitian metric on $\mathcal{O}(D)$. We denote by $\nabla_D:=\nabla_{i}^*\otimes\nabla_{h_0}\otimes\nabla_{h_D'}$ a $h_0^*\otimes h_0\otimes h_D'$-connection on $\Hom(E_i,E)\otimes \mathcal{O}(D)$.
If we use the identity $\Box^{\partial^{\nabla_D}}_{\omega_i}=\Box^{\delbar^{\nabla_D}}_{\omega_i}-\sqrt{-1}[\Lambda_{\omega_i},F_{\nabla_D}]$, the following calculation works:
\begin{align}\label{estimate on difference}
&\int_X\langle \nabla_D(s_D^{2m}\Phi_i),\nabla_D(s_D^{2m}\Phi_i)\rangle_{h_0^*\otimes h_0\otimes h_D'}\omega_i^n\notag\\
&=\int_X\langle\Delta^{\nabla_D}_{\omega_i}(s_D^{2m}\Phi_i),s_D^{2m}\Phi_i\rangle_{h_0^*\otimes h_0\otimes h_D'}\omega_i^n \notag\\
&=-\int_X\langle\sqrt{-1}\Lambda_{\omega_i}F_{\nabla_D}\cdot (s_D^{2m}\Phi_i), s_D^{2m}\Phi_i\rangle_{h_0^*\otimes h_0\otimes h_D'}\omega_i^n \notag\\
&=\int_X\langle-\sqrt{-1}\Lambda_{\omega_i}F_{\nabla_{h_0}\otimes\nabla_{h_D'}}\circ(s_D^{2m}\Phi_i),s_D^{2m}\Phi\rangle_{h_0^*\otimes h_0\otimes h_D'}\omega_i^n\notag\\
&\hspace{4mm}+\int_X\langle(s_D^{2m}\Phi_i)\circ(\sqrt{-1}\Lambda_{\omega_i}F_{\nabla_i}),s_D^{2m}\Phi_i\rangle_{h_0^*\otimes h_0\otimes h_D'}\omega_i^n\notag\\
&\le C\int_X\langle|s_D|^{2m}\sqrt{-1}\Lambda_{\omega_i}\omega_0\cdot(s_D^m\Phi_i),s_D^m\Phi_i\rangle_{h_0^*\otimes h_0\otimes h_D}\omega_i^n\notag\\
&\hspace{4mm}+|\lambda_i|\int_X|s_D^{2m}\Phi_i|_{h_0^*\otimes h_0\otimes h_D'}^2\omega_i^n\notag\\
&\le C.
\end{align}
\end{proof}

\subsection{proof of the main result}\label{pf of KH corr}
We work under Assumption \ref{assumption4}.
We remember that the underlying smooth hermitian vector bundle $\underline{E_{\infty}}$ of $E_{\infty}$ coincides with the underlying hermitian vector bundle $\underline{E}|_{X\setminus Z}$ by \cite{CW22} (refer to Corollary \ref{Uhlenbeck compactness}).
Furthermore, as a smooth section of $\Hom(\underline{E_{\infty}},\underline{E}|_{X\setminus Z})=\Hom(\underline{E}|_{X\setminus Z},\underline{E}|_{X\setminus Z})=\Hom(\underline{E}|_{X\setminus Z}, \underline{E_{\infty}})$, smooth endomorphisms $\Phi_{\infty}$ and $\Psi_{\infty}$ are both hermitian.
We first show the following:
\begin{prop}\label{adHYM}
Assume that $\Psi_{\infty}:E|_{X\setminus D}\to E_{\infty}$ is isomorphic. Let $h_0$ be a smooth hermitian metric on $\underline{E}$ and $\nabla_{\infty}$ be the $T$-admissible HYM connection on $(E_{\infty},h_0)$. Then $h_{\infty}:=\Psi_{\infty}^*h_0$ is a $T$-admissible HYM metric on $E$ whose Chern connection is $\Psi_{\infty}^{-1}\circ\nabla_{\infty}\circ\Psi_{\infty}$.
\end{prop}
\begin{proof}
We can see that $\Psi_{\infty}^{-1}\circ\nabla_{\infty}\Psi_{\infty}$ is the Chern connection of $h_{\infty}$ by the similar way with Lemma \ref{construction of HYM conns}. Therefore it suffices to prove that the curvature tensor $F_{h_{\infty}}$ is $L^2$-integrable with respect to the norm defined by $h_{\infty}$ and $T$. We remark that $F_{h_{\infty}}=\Psi_{\infty}^{-1}\circ F_{\nabla_{\infty}}\circ \Psi_{\infty}$ and $\sqrt{-1}F_{h_{\infty}}^{*_{h_{\infty}}}=\sqrt{-1}F_{h_{\infty}}$. Then, the point-wise identity $|F_{h_{\infty}}|_{h_{\infty},T}^2T^n=\Tr(\sqrt{-1}F_{h_{\infty}}\wedge \sqrt{-1}F_{h_{\infty}})\wedge T^n+|\Lambda_TF_{h_{\infty}}|_{h_{\infty}}^2T^n$ on $X\setminus Z$ (c.f. \cite{Kob} Chapter 4 section 4) shows $|F_{h_{\infty}}|_{h_{\infty},T}^2=|F_{\nabla_{\infty}}|_{h_0,T}^2$ on $X\setminus Z$. Thus $h_{\infty}$ is a $T$-admissible HYM metric on $E|_{X\setminus D}$.
\end{proof}
Then we provide the proof in the case $E=\mathcal{O}(ND)$ for large $N>0$ where $D=E_{nK}(\alpha)$.
\begin{prop}\label{KH corr on nK locus}
We work under Assumption \ref{assumption4}.\\
If $N\ge3m+1$ and set $E=\mathcal{O}(ND)$, then $\Psi_{\infty}:\mathcal{O}(ND)|_{X\setminus D}\to \mathcal{O}(ND)_{\infty}$ is isomorphic. 
\end{prop}
\begin{proof}
Let us consider a composition of sheaf morphisms
$$
\mathcal{O}_{X\setminus D}\xrightarrow{s_D^N} \mathcal{O}(ND)\xrightarrow{\Psi_{\infty}} \mathcal{O}(ND)_{\infty}.
$$
We denote by $h_0$ a reference hermitian metric on the underlying smooth line bundles $\underline{\mathcal{O}(ND)}$ and $\underline{\mathcal{O}(ND)_{\infty}}=\underline{\mathcal{O}(ND)}|_{X\setminus Z}$. We give the Euclidean metric $h_1$ on $\mathcal{O}_{X}$. Then $\Psi_{\infty}\circ s_D^N\in \Gamma(X\setminus Z, \Hom(\mathcal{O}_X,\mathcal{O}(ND)_{\infty}))$ is holomorphic with respect to $\nabla_{\infty}\otimes \nabla_{h_1}^*$. Hence, if we recall that $\mu_{\alpha}(\mathcal{O}(ND))=0$, we obtain the following equation on $X\setminus Z$ by the direct calculation:
\begin{equation}\label{KH corr on nK locus eq1}
\Delta_T|\Psi_{\infty}\circ s_D^N|^2=|(\nabla_{\infty}\otimes\nabla_{h_1}^*)(\Psi_{\infty}\circ s_D^N)|_T^2.
\end{equation}
By proposition \ref{uniform esti of diff}, the RHS is integrable on $X\setminus Z$. Let us show that the integral of the LHS equals to 0. We denote by $\nabla_D$ the Chern connection of a smooth hermitian metric $h_0$ on $\mathcal{O}(ND)$. We denote by $Z^{\varepsilon}$ an ${\varepsilon}$-neighborhood of $Z$.  Let $\eta_{\varepsilon}:X\to \mathbb{R}$ be a cut-off function such that $\Supp(\eta_{\varepsilon})\subset X\setminus Z^{\varepsilon/2}$, $\eta_{\varepsilon}\equiv 1$ on  $X\setminus Z^{\varepsilon}$ and $|\nabla\eta_{\varepsilon}|_{\omega_0}<1/\varepsilon$. We first compute as follows:
\begin{align*}
\left|\nabla|\Psi_{\infty}\circ s_D^N|^2\right|_T
&=2\left|\langle (\nabla_{\infty}\otimes\nabla_{h_1}^*)(\Psi_{\infty}\circ s_D^N), \Psi_{\infty}\circ s_D^N\rangle\right|_T\\
&=2\left|\langle \nabla_{\infty}\Psi_{\infty}\circ s_D^N+\Psi_{\infty}\circ(Ns_D^{N-1}\nabla_Ds_D), \Psi_{\infty}\circ s_D^N\rangle\right|_T\\
&\le 2|s_D^{2m}\nabla_{\infty}\Psi_{\infty}|_T|s_D^m\Psi_{\infty}||s_D^{2N-3m}|
+2N|s_D^m\Psi_{\infty}|^2|s_D^{2N-2m-1}||\nabla_Ds_D|\\
&\le C|s_D^{2m}\nabla_{\infty}\Psi_{\infty}|_T|s_D^{2N-3m}|+C|s_D^{2N-2m-1}|.
\end{align*}
In the fourth line, we used $\|s_D^m\Psi_{\infty}\|_{L^{\infty}(X\setminus Z)}<\infty$ (refer to Corollary \ref{L-infty estimate of limit object}).
Then, if we remark that $\eta_{\varepsilon}$ is compactly supported in $X\setminus Z$ and $T\ge |s_D|^{2m}\omega_0$, we can calculate as follows:
\begin{align}\label{KH corr on nK locus eq3}
\left|\int_{X\setminus Z}\eta_{\varepsilon}\Delta_T|\Psi_{\infty}\circ s_D^N|^2T^n\right|
&=\left|\int_{X\setminus Z}\langle \nabla\eta_{\varepsilon},\nabla|\Psi_{\infty}\circ s_D^N|^2\rangle_TT^n\right|\notag\\
&\le \int_{X\setminus Z} |\nabla\eta_{\varepsilon}|_T|\nabla|\Psi_{\infty}\circ s_D^N|^2|_TT^n\notag\\
&\le C\int_{X\setminus Z}(|\nabla\eta_{\varepsilon}|_{\omega_0}|s_D|^{-2m})(|s_D^{2m}\nabla_{\infty}\Psi_{\infty}|_T|s_D^{2N-3m}|+|s_D^{2N-2m-1}|)T^n\notag\\
&\le C \int_{X\setminus Z}|\nabla\eta_{\varepsilon}|_{\omega_0}|s_D^{2m}\nabla_{\infty}\Psi_{\infty}|_T|s_D^{2N-5m}|T^n + C\int_{X\setminus Z}|\nabla\eta_{\varepsilon}||s_D^{2N-4m-1}|T^n\notag\\
&\le C\left(\int_{Z^{\varepsilon}\setminus Z^{\varepsilon/2}}|\nabla\eta_{\varepsilon}|^2_{\omega_0}|s_D|^{2m+4}e^F\omega_0^n \right)^{1/2}\left(\int_{X\setminus Z}|s_D^{2m}\nabla_{\infty}\Psi_{\infty}|_T^2T^n \right)^{1/2}\notag \\
&\hspace{4mm}+C\int_{X\setminus Z}|\nabla\eta_{\varepsilon}|_{\omega_0}|s_D|^{2m+1}e^F\omega_0^n.
\end{align}
Since $|\nabla\eta_{\varepsilon}|^2_{\omega_0}|s_D|^{2m+4}e^F\le C$ around $D$ and $Z\setminus D=\Sigma$ has codimension 2, we obtain
$\lim_{\varepsilon\to 0}\int_X|\nabla\eta_{\varepsilon}|_{\omega_0}^2\omega_0^n=0$. Then, by Proposition \ref{uniform esti of diff}, we obtain that
\begin{equation}\label{KH corr on nK locus eq2}
\int_{X\setminus Z}\Delta_T|\Psi_{\infty}\circ s_D^N|^2T^n=\lim_{\varepsilon\to 0}\int_{X\setminus Z}\eta_{\varepsilon}\Delta_T|\Psi_{\infty}\circ s_D^N|^2T^n=0.
\end{equation}
By (\ref{KH corr on nK locus eq1}) and (\ref{KH corr on nK locus eq2}), we obtain $\int_{X\setminus Z}|(\nabla_{\infty}\otimes\nabla_{h_1}^*)\Psi_{\infty}\circ s_D^N|_T^2T^n=0$. In particular $(\nabla_{\infty}\otimes\nabla_{h_1}^*)(\Psi_{\infty}\circ s_D^N)=0$ on $X\setminus Z$.
Hence the rank of the image of $\Psi_{\infty}\circ s_D^N$ is constant on $X\setminus Z$. Since $\Psi_{\infty}\ne 0$, we obtain that $\Psi_{\infty}\circ s_D^N: \mathcal{O}_{X\setminus Z}\to \mathcal{O}(ND)_{\infty}$ is isomorphic on $X\setminus Z$.
 Since $s_D$ has no zero on $X\setminus Z$, we can conclude that $\Psi_{\infty}:\mathcal{O}(ND)\to \mathcal{O}(ND)_{\infty}$ is isomorphic on $X\setminus Z$. Now we recall that $\Sigma\subset X\setminus D$ is an analytic subset of codimension at least 2. Thus $\Psi_{\infty}$ is isomorphic on $X\setminus D$. 
\end{proof}
We obtain the existence of a $T$-admissible HYM connection in rank 1 case.
\begin{theo}\label{KH corr for rank1}
We work under Assumption \ref{assumption4}.
Let $X$ be a compact K\"{a}hler manifold and $\alpha$ be a nef and big class on $X$ such that $D=E_{nK}(\alpha)$ is a snc divisor. Let $T$ be a closed positive $(1,1)$-current in $\alpha$ which satisfies Assumption \ref{assumption3}. Then, for any holomorphic line bundle $E$ on $X$, the sheaf morphism $\Psi_{\infty}:E|_{X\setminus D}\to E_{\infty}$ is isomorphic. 
\end{theo}
\begin{proof}
Fix a natural number $N\ge 3m+1$ and denote by $\Psi_{D,\infty}:\mathcal{O}(2ND)|_{X\setminus D}\to\mathcal{O}(2ND)_{\infty}$ the isomorphism in Proposition \ref{KH corr on nK locus}. Since $\Psi_{\infty}$ is a holomorphic nonzero section of $E^*\otimes E_{\infty}$ and $\Phi_{\infty}$ is a holomorphic nonzero section of $E\otimes E_{\infty}^*$, we naturally obtain a holomorphic function $\Psi_{\infty}(\Phi_{\infty})$ on $X\setminus D$. If we consider a sheaf morphism 
$$
\mathcal{O}_{X}\xrightarrow{s_D^{2N}} \mathcal{O}(2ND)\xrightarrow{\Psi_{D,\infty}}\mathcal{O}(2ND)_{\infty},
$$
we obtain a holomorphic nonzero global section $s:=\Psi_{D,\infty}\left(s_D^{2N}(\Psi_{\infty}(\Phi_{\infty}))\right)\in H^0(X\setminus D,\mathcal{O}(2ND)_{\infty})$. 
We give a connection on each line bundle by 
\begin{itemize}
\item the Chern connection $\nabla_0$ of $h_0$ on $E$,
\item the $T$-admissible HYM $h_0$-connection $\nabla_{\infty}$ on $E_{\infty}$,
\item the Chern connection $\nabla_{D}$ of a smooth hermitian metric $h_D^{2N}$ on $\mathcal{O}(2ND)$ and 
\item the $T$-admissible HYM $h_D^{2N}$-connection $\nabla_{D,\infty}$ on $\mathcal{O}(2ND)_{\infty}$.
\end{itemize}
For simplicity, we denote $\nabla_0\otimes\nabla_{\infty}^*$ on $E\otimes E_{\infty}^*$ and $\nabla_0^*\otimes\nabla_{\infty}$ on $E^*\otimes E_{\infty}$ by the same notation $\nabla_{\infty}$. We also denote the connection $\nabla_D^*\otimes\nabla_{D,\infty}$ on $\mathcal{O}(2ND)^*\otimes\mathcal{O}(2ND)_{\infty}$ by $\nabla_{D,\infty}$. Then, as the previous Proposition \ref{KH corr on nK locus}, we have
$$
\Delta_T|s|^2=|\nabla_{D,\infty}s|_T^2-\mu_{\alpha}(\mathcal{O}(2ND))|s|^2=|\nabla_{D,\infty}s|_T^2.
$$
If we remark that $|s_D^m\Psi_{D,\infty}|$, $|s_D^m\Psi_{\infty}|$ and $|s_D^m\Phi_{\infty}|$ are contained in $L^{\infty}(X\setminus Z)$ (refer to Corollary \ref{L-infty estimate of limit object}), the following calculation works at each point on $X\setminus Z$:
\begin{align*}
|\nabla_{D,\infty}s|_T
&\le |\nabla_{D,\infty}\Psi_{D,\infty}||s_D^{2N}||\Psi_{\infty}||\Phi_{\infty}|
+|\Psi_{D,\infty}||Ns_D^{2N-1}\nabla_Ds_D|_T|\Psi_{\infty}||\Phi_{\infty}|\\
&\hspace{4mm} +|\Psi_{D,\infty}||s_D|^{2N}(|\nabla_{\infty}\Psi_{\infty}|+|\nabla_{\infty}\Phi_{\infty}|)\\
&\le |s_D^{2m}\nabla_{D,\infty}\Psi_{D,\infty}||s_D^{m}\Psi_{\infty}||s_D^m\Phi_{\infty}||s_D^{2N-4m}|+N|s_D^m\Phi_{D,\infty}||\nabla_Ds_D|_{\omega_0}|s_D^m\Psi_{\infty}||s_D^m\Phi_{\infty}||s_D^{2N-1-5m}|\\
&\hspace{4mm} +|s_D^{m}\Psi_{D,\infty}|(|s_D^{2m}\nabla_{\infty}\Psi_{\infty}|_T+|s_D^{2m}\nabla_{\infty}\Phi_{\infty}|_T)|s_D^{2N-3m}|\\
&\le |s_D|^{2N-6m}\left(C|s_D^{2m}\nabla_{D,\infty}\Psi_{D,\infty}|_T+C+C|s_D^{2m}\nabla_{\infty}\Psi_{\infty}|_T+C|s_D^{2m}\nabla_{\infty}\Phi_{\infty}|_T\right).
\end{align*}
Since each term in RHS is $L^2$-integrable with respect to $T^n$ by Proposition \ref{uniform esti of diff}, we can see that $|\nabla_{D,\infty}s|_T$ is $L^2$-integrable with respect to $T^n$. Then, the same calculation with (\ref{KH corr on nK locus eq3}) shows 
$$
0=\int_{X\setminus Z}\Delta_T|s|^2T^n=\int_{X\setminus Z}|\nabla_{D,\infty}s|_TT^n.
$$
In particular, the rank of the image of $\cdot s: \mathcal{O}_{X\setminus Z}\to \mathcal{O}(ND)_{\infty}$ is constant. Then, since $s_D$ and $\Psi_{D,\infty}$ has no zero point on $X\setminus D$ by Proposition \ref{KH corr on nK locus} and the rank of each vector bundle equals to 1, we obtain that $\Psi_{\infty}(\Phi_{\infty})$ has no zero point on $X\setminus Z$. As a consequence, the sheaf morphism $\Psi_{\infty}:E|_{X\setminus D}\to E_{\infty}$ is isomorphic.
\end{proof}
Then we can prove in general rank.
\begin{theo}\label{KH corr nef and big} 
We work under Assumption \ref{assumption4}.
Let $X$ be a compact K\"{a}hler manifold and $\alpha$ be a nef and big class on $X$ such that $D=E_{nK}(\alpha)$ is a snc divisor. Let $T$ be a closed positive $(1,1)$-current in $\alpha$ which satisfies Assumption \ref{assumption3}. If $E$ is an $\alpha^{n-1}$-slope stable holomorphic vector bundle on $X$, then $\Psi_{\infty}:E|_{X\setminus D}\to E_{\infty}$ is an isomorphism on $X\setminus D$. 
\end{theo}
\begin{proof}
Let us consider a sheaf morphism $\Psi_{\infty}:E|_{X\setminus D}\to E_{\infty}$. We firstly prove that the rank of the image of $\Psi_{\infty}$ coincides with $r=\rk E_{\infty}$. Suppose $\rk(\Psi_{\infty})<r$. Then we obtain the following exact sequence of coherent sheaves on $X\setminus D$:
$$
0\to \mathcal{F}\to E \xrightarrow{s_D^N\Psi_{\infty}} \mathcal{G} \to 0
$$
where $\mathcal{G}\subset E_{\infty}\otimes \mathcal{O}(ND)$ is the image sheaf of $\rk \mathcal{G}=q<r$.
\begin{claim}
Above $\mathcal{F}$ and $\mathcal{G}$ extend to coherent sheaves on $X$.
\end{claim}
\begin{proof}
Let $j: X\setminus D\hookrightarrow X$ be the natural inclusion. Then $s_D^N\Psi_{\infty}: E|_{X\setminus D}\to E_{\infty}$ naturally defines an $\mathcal{O}_X$-linear morphism $\psi: E\to j_*E_{\infty}$ by $\psi(e)=s_D^N\Psi_{\infty}(e|_{X\setminus D})$ for any local section $e$ of $E$. Let us define $\widetilde{\mathcal{G}}:={\rm{Im}}(\psi)$. Then $\widetilde{\mathcal{G}}$ is a finitely generated $\mathcal{O}_X$-module. We have a surjection $\psi: E\to \widetilde{\mathcal{G}}$. Hence its dual gives an injection $\psi^*:\widetilde{\mathcal{G}}^*\hookrightarrow E^*$. Hence $\widetilde{\mathcal{G}}^*\simeq {\rm{Im}}(\psi^*)$ is a finitely generated $\mathcal{O}_X$-submodule of a locally free sheaf $E$. Then, by Oka's coherence theorem (c.f.\cite{Kob} Lemma 5.3), $\widetilde{\mathcal{G}}^*$ is a coherent sheaf on $X$. We remark that $\widetilde{\mathcal{G}}$ is a torsion free $\mathcal{O}_X$-module. In fact, $\widetilde{\mathcal{G}}$ is $j_*\mathcal{O}_{X\setminus D}$-torsion free and there is a natural inclusion $\mathcal{O}_X\hookrightarrow j_*\mathcal{O}_{X\setminus D}$. Hence we have the natural inclusion $\widetilde{\mathcal{G}}\hookrightarrow\widetilde{\mathcal{G}}^{**}$. Since $\widetilde{\mathcal{G}}^{**}$ is a coherent reflexive sheaf on $X$, we have a local injection $\widetilde{\mathcal{G}}^{**}\hookrightarrow \mathcal{O}^{\oplus M}$. Therefore $\widetilde{\mathcal{G}}$ is locally a finitely generated $\mathcal{O}_X$-submodule of $\mathcal{O}^{\oplus M}$ and thus $\widetilde{\mathcal{G}}$ is a coherent sheaf on $X$ again by Oka's theorem.
\end{proof}
By the $\alpha^{n-1}$-stability of $E$, we have $\mu_{\alpha}(E)<\mu_{\alpha}(\mathcal{G})$.
We consider the following diagram:

\begin{center}
\begin{tikzpicture}[auto]
\node (e) at (0,2) {$\bigwedge^qE|_{X\setminus D}$}; \node (einfty) at (4,2) {$\bigwedge^qE_{\infty}\otimes \mathcal{O}(ND)$}; \node (g) at (2,0) {$\bigwedge^q\mathcal{G}/\Tor|_{X\setminus D}$};
\draw[->>] (e) to node[swap] {$f$} (g); \draw[->] (e) to node {$s_D^N\Psi_{\infty}$} (einfty); 
\draw[->] (g) to node[swap] {$s$} (einfty); 
\end{tikzpicture}
\end{center}

Via the inclusion $\bigwedge^q\mathcal{G}/\Tor\hookrightarrow \left(\bigwedge^q\mathcal{G}/\Tor\right)^{**}=\det\mathcal{G}$, we restrict  a smooth hermitian metric $h_{\mathcal{G},0}$ on $\det\mathcal{G}$ to $\bigwedge^q\mathcal{G}/\Tor$. 
\begin{claim}\label{bdd of s}
$|s|_{h_{\mathcal{G},0}^*\otimes h_{0}\otimes h_{D,0}^N}$ is globally bounded on the locally free locus of $\bigwedge^q\mathcal{G}/\Tor$. 
\end{claim}
\begin{proof}[proof of claim]
Let $U$ be an open set in $X$ and $e_1,\ldots,e_k\in \Gamma(U,\bigwedge^qE)$ be a local holomorphic frame. Then, for any $p\in U\setminus \Sing(\mathcal{G})$, there exists $i$ such that $f(e_i)$ gives a local frame of $\bigwedge^q\mathcal{G}/\Tor$ around $p$ since $f: \bigwedge^qE\to \bigwedge^q\mathcal{G}/\Tor$ is surjective. 
On $X\setminus \Sing(\mathcal{G})$, the morphism $s:\bigwedge^q\mathcal{G}/\Tor|_{X\setminus D}\to\bigwedge^qE_{\infty}\otimes\mathcal{O}(ND)_{\infty}$ is the natural inclusion given by $\mathcal{G}\subset  E_{\infty}\otimes\mathcal{O}(ND)_{\infty}$. Hence $f(e_i)=s_D^N\Psi_{\infty}(e_i)$ on $X\setminus (D\cup\Sing(\mathcal{G}))$.
We remark that $(f(e_i))^{-1}$ is a local frame of $\left(\bigwedge^q\mathcal{G}/\Tor\right)^*$.
Let us locally denote by $s=\tilde{s}\cdot(f(e_i))^{-1}$. Obviously we have $\tilde{s}=s(f(e_i))=s_D^N\Psi_{\infty}(e_i)$. Then, since hermitian metrics $h_{\mathcal{G},0}$, $h_0$ and $h_{D,0}$ are smooth and bounded on $X$, we have the following estimate on $X\setminus (D\cup\Sing(\mathcal{G}))$:
$$
|s|_{h_{\mathcal{G},0}^*\otimes h_{0}\otimes h_{D,0}^N}=|\tilde{s}\cdot f(e_i)^{-1}|_{h_{\mathcal{G},0}^*\otimes h_{0}\otimes h_{D,0}^N}\le C|s_D^N\Psi_{\infty}(e_i)\cdot (s_D^N\Psi_{\infty}(e_i))^{-1}|\le C.
$$ 
\end{proof}
Let us consider $\Psi_{D, \infty}:\mathcal{O}((N+2m)D)\to \mathcal{O}((N+2m)D)_{\infty}$ and $\Phi_{\mathcal{G},\infty}:(\det\mathcal{G})_{\infty}\to \det\mathcal{G}$ on $X\setminus D$. These are isomorphic by Theorem \ref{KH corr for rank1}. By Corollary \ref{L-infty estimate of limit object}, we also have $s_D^m\Phi_{\mathcal{G},\infty}$ and $s_D^m\Psi_{D,\infty}$ are uniformly bounded on $X\setminus D$. Let us consider an isomorphism $\Id_{\bigwedge^qE_{\infty}}\otimes(\Psi_{D,\infty}\circ s_D^{2m})\otimes\Phi_{\mathcal{G}_{\infty}}^*:\bigwedge^qE_{\infty}\otimes\mathcal{O}(ND)\otimes(\det\mathcal{G})^*\to \bigwedge^qE_{\infty}\otimes\mathcal{O}((N+2m)D)_{\infty}\otimes(\det\mathcal{G})_{\infty}^*$. Then the morphism $s\in H^0(X\setminus D, \bigwedge^qE_{\infty}\otimes\mathcal{O}(ND)\otimes(\det\mathcal{G})^*)$ in Claim \ref{bdd of s} defines $s':=\Psi_{D,\infty}\circ(s_D^{2m}s)\circ\Phi_{\mathcal{G},\infty}\in H^0(X\setminus D,\bigwedge^qE_{\infty}\otimes\mathcal{O}((N+2m)D)_{\infty}\otimes(\det\mathcal{G})_{\infty}^*)$.
Then $|s_D^{2m}s|_{h_0\otimes h_{D,\infty}\otimes h_{\mathcal{G},\infty}}$ is uniformly bounded. In fact, we can see the boundedness by
$$
|s'|_{h_0\otimes h_{D,0}\otimes h_{\mathcal{G},0}^*}
\le |s_D^m\Psi_{D,\infty}|_{h_{D,0}}|s_D^m\Phi_{\mathcal{G},\infty}|_{h_{D,0}\otimes h_{\mathcal{G},0}}|s|_{h_{\mathcal{G},0}^*\otimes h_0\otimes h_{D,0}^N}
$$
and Claim \ref{bdd of s} and Corollary \ref{L-infty estimate of limit object}. Therefore, by choosing $N$ large, we can assume $|s'|(x_j)\to0$ for any sequence $x_j\in X\setminus D$ which converges to a point $x\in D$. We then consider 
\begin{align}\label{KH corr nef and big eq1}
\Delta_T\log(|s'|_{h_0\otimes h_{D,0}\otimes h_{\mathcal{G},0}}+1)
&\ge -q(\mu_{\alpha}(E)-\mu_{\alpha}(\mathcal{G}))> 0.
\end{align}
Since $T$ is smooth K\"{a}hler on $X\setminus D$ and each vector bundle is $T$-HYM, we see that\\
 $\log(|s'|^2_{h_0\otimes h_{D,0}\otimes h_{\mathcal{G},0}}+1)$ is $L^2_1$-local on $X\setminus D$ by \cite{Chen} Proposition 2.17. Then, by the weak maximum principle (c.f. \cite{GilT} Theorem 8.1), we obtain that 
$$
\sup_{X\setminus D}\log(|s'|_{h_0\otimes h_{D,0}\otimes h_{\mathcal{G},0}}+1)=\sup_D\log(|s'|_{h_0\otimes h_{D,0}\otimes h_{\mathcal{G},0}}+1)=0.
$$
It contradicts to (\ref{KH corr nef and big eq1}). Hence $\Psi_{\infty}:E\to E_{\infty}$ has to be full rank.\\
Then we can see that $\Psi_{\infty}$ is isomorphic. In fact, $\det(\Psi_{\infty}):\det E\to \det E_{\infty}$ gives a nontrivial sheaf morphism. Let us consider a nontrivial sheaf morphism $\Phi_{\infty}:\det E_{\infty}\to \det E$. Then, by the proof of Theorem \ref{KH corr for rank1}, we can see that a holomorphic function $\det\Psi_{\infty}(\Phi_{\infty})$ on $X\setminus D$ has no zero point on $X\setminus D$. In particular $\det \Psi_{\infty}$ is nowhere vanish on $X\setminus D$, which implies $\Psi_{\infty}$ is isomorphic.
\end{proof}
By replacing $T$ by $\pi^*\omega$ in Assumption \ref{assumption2}, we obtain our main result:
\begin{theo}\label{main thm thm}Let $(Y,\omega_Y)$ be a compact K\"{a}hler manifold.
Let $\omega$ be the solution to the complex Monge-Amp\`{e}re equation $\omega^n=e^F\omega_Y^n$, where $F$ satisfies the following conditions:
\begin{itemize}
\item there is an analytic subset $Z\subset Y$ with $\codim Z\ge 2$ such that $F\in C^{\infty}(Y\setminus Z)$ and
\item there is a constant  $0\le a \le 1/2$ such that $F=-a\log(|f_1|^2+\cdots+|f_l|^2)+O(1)$ locally where $(f_1,\ldots,f_l)$ is a local defining functions of $Z$. 
\end{itemize}
Then, if a holomorphic vector bundle $F$ on $Y$ is $\{\omega\}^{n-1}$-slope polystable, then $F$ admits an $\omega$-admissible HYM metric.
\end{theo}

\end{document}